\documentclass[11pt]{article}

\usepackage{amstext}
\usepackage{amsmath}
\usepackage{amssymb}
\usepackage{amsthm}
\usepackage{a4}
\usepackage{enumerate}

\newtheorem{theo}{Theorem}[section]
\newtheorem{lemma}[theo]{Lemma}
\newtheorem{cor}[theo]{Corollary}

\theoremstyle{definition}
\newtheorem*{rem}{Remark}
\newtheorem*{ex}{Example}

\newcommand{\N}{\mathbb{N}}

\newcommand{\R}{\mathbb{R}}

\newcommand{\eps}{\varepsilon}

\newcommand{\weakly}{\rightharpoonup}

\newcommand{\pr}{{\rm pr}}
\newcommand{\supp}{{\rm supp}}
\newcommand{\dist}{{\rm dist}}

\begin{document}

\begin{center}
\begin{Large}
On the infinite particle limit in Lagrangian dynamics and convergence of optimal transportation meshfree methods
\end{Large}
\end{center}

\begin{center}
\begin{large}
Bernd Schmidt\\[0.2cm]
\end{large}
\begin{small}
Institut f{\"u}r Mathematik,\\ 
Universit{\"a}t Augsburg\\ 
86135 Augsburg, Germany\\ 
{\tt bernd.schmidt@math.uni-augsburg.de}
\end{small}
\end{center}

\begin{abstract}
We consider Lagrangian systems in the limit of infinitely many particles. It is shown that the corresponding discrete action functionals Gamma-converge to a continuum action functional acting on probability measures of particle trajectories. Also the convergence of stationary points of the action is established. Minimizers of the limiting functional and, more generally, limiting distributions of stationary points are investigated and shown to be concentrated on orbits of the Euler-Lagrange flow. We also consider time discretized systems. These results in particular provide a convergence analysis for optimal transportation meshfree methods for the approximation of particle flows by finite discrete Lagrangian dynamics. 
\end{abstract}

\section{Introduction}

In classical Lagrangian mechanics a system of particles is described by an action functional on the particle trajectories. According to Hamilton's principle the dynamics of the system is given by stationary points of this functional. In continuum mechanics, on the other hand, a medium such as, e.g., a fluid is described by its mass density and fields for physical observables whose dynamics are governed by partial differential equations such as the velocity obeying, e.g., Euler's equation of motion. While at microscopic molecular length scales a fluid cannot be viewed as a homogeneous medium, these fields are assumed to describe material properties in a suitably mesoscopically averaged sense. The natural question therefore arises, if these different models can be related to each other. 

In order to address this question, one is naturally led to extend the set-up of Lagrangian mechanics to infinite dimensional systems and to devise action functionals acting on mass densities or, more generally and after normalization, probability measures that describe the mass distribution of the system. If with the help of a single particle Lagrangian a `Lagrangian cost function' is defined as the minimal value of the action necessary to move a particle from one point in space to another, a variational framework for such systems can be given within the theory of optimal transportation. (See, e.g., \cite{Villani:09} for a recent account on optimal transportation theory with Lagrangian costs.)

A connection between mass transportation problems and fluid mechanics has been derived in the seminal work \cite{BenamouBrenier:99,BenamouBrenier:00} of Benamou and Brenier (see also \cite{Villani:03}), who show that the Wasserstein distance between two probability densities can be expressed as a minimal action value with Lagrangian cost, when minimized over velocity fields constrained to satisfy the continuity equation ensuring conservation of mass. Departing from this relation, Li, Habbal and Ortiz have constructed an efficient Lagrangian meshfree approximation scheme for fluid flows by material point sampling and solving for discrete minimizers of an approximating finite dimensional Lagrangian system, see \cite{LiHabbalOrtiz}. An interesting aspect of this approach is that, due to their variational structure, the discrete approximating systems can be formulated within the theory of variational integrators and consequently possess good conservation properties of physically conserved quantities. See, e.g., \cite{MarsdenWest} and the references therein for a general introduction to the theory of variational integrators and \cite{SchmidtLeyendeckerOrtiz:09} for a convergence analysis on manifolds. As discussed in \cite{LiHabbalOrtiz}, conforming fields may be interpolated efficiently with max-ent shape functions as developed in \cite{ArroyoOrtiz}, for which convergence has been proved in the recent article \cite{BompadreSchmidtOrtiz}. 

The main aim of the present work is to provide a rigorous variational analysis of the infinite particle limit in Lagrangian mechanics. Such an analysis appears to be interesting from at least three different perspectives: 

Firstly our results supply a rigorous convergence analysis of the optimal transportation meshfree method constructed in \cite{LiHabbalOrtiz}. We refer to \cite{LiHabbalOrtiz} for an extensive description of this method, its comparison with alternative numerical integration schemes and numerical experiments. 

Secondly our approximation scheme also provides a novel method to infer characteristic properties of optimal transport maps and their displacement interpolation with Lagrangian costs. In particular we show that with the help of our discrete-to-continuum approximation scheme through $\Gamma$-convergence it is possible to re-derive a result of Bernard and Buffoni \cite{BernardBuffoni:07} characterizing dynamical optimal transference plans. 

Thirdly the problem is of some independent interest as it provides an effective theory derived by a rigorous variational discrete-to-continuum $\Gamma$-limit for a time dependent problem. Such problems have received a lot of attention over the last years in particular for static problems, see, e.g., the survey article \cite{BlancLeBrisLions} by Blanc, LeBris and Lions. It has to be noted, however, that the much more interesting problem for interacting particles is beyond the scope of this paper and deferred to future work. 

Having reviewed some basic material on Lagrangian systems for single and many particle systems in Section \ref{section:Lagrangian-systems}, in Section \ref{section:Limits-Gamma} we state and prove the main variational convergence results in the limit of infinitely many particles: The discrete action functionals $\Gamma$-converge to a continuum action functional acting on probability measures of particle trajectories and discrete action minimizers converge to action minimizers of the limiting continuum functional. In view of our numerical convergence results in Section \ref{section:Numerics} for systems as in \cite{LiHabbalOrtiz}, we note that our assumptions allow for maximal flexibility in sampling approximating marginal measures as we do not assume that these discrete measures are derived by some particular coarse graining procedure. 

As an application, Section \ref{section:Limits-Analysis} provides an analysis of the limiting continuum mass transportation problem with Lagrangian cost: Minimizers of the limiting functional and, more generally, limiting distributions of stationary points are investigated and shown to be concentrated on minimizing orbits of the Euler-Lagrange flow. Section \ref{section:Limits-Stat-points} then extends the convergence properties from action minimizers to general stationary points of the action. 

In Section \ref{section:Unbounded-energy}, the $\Gamma$-convergence result is extended to systems with unbounded potential energies for small time intervals. Under suitable conditions, convergence of stationary points of the action is proved even on long time intervals. 

Section \ref{section:Numerics} finally gives a rigorous convergence analysis of the aforementioned optimal transportation meshfree method, where the discrete trajectories are computed through a numerical quadrature formula (the basic midpoint rule). This is achieved by extending the previous results to a time discretized set-up and proving convergence when simultaneously the time step converges to $0$ and the number of particles diverges to $\infty$.

\section{Lagrangian systems}\label{section:Lagrangian-systems}

This section serves to collect some notation and well-known basics on Lagrangian systems. Proofs of these results can be found, e.g., in \cite{Fathi} or \cite{Villani:09}.

\subsection{Single particle systems}

Let $(M, g)$ be a connected complete Riemannian manifold of class $C^{\infty}$. By $TM = \bigcup_{x \in M} T_x M$, $T^*M = \bigcup_{x \in M} T_x^* M$, $\pi_M : TM \to M$ and $\pi_M^* : T^*M \to M$ we denote the tangent and cotangent bundle and their canonical projections onto $M$, respectively. We assume that there is a (time independent) Lagrangian $L : TM \to \R$, $(x, v) \mapsto L(x, v)$, satisfying the following set of classical conditions: 
\begin{enumerate}[(i)]
\item Smoothness: $L \in C^{\infty}(TM; \R)$, 
\item\label{item:strict-convexity} Strict convexity: For each compact set $K \subset M$ there exists some positive constant $c_0$ such that $\nabla^2_v L(x, \cdot) \ge c_0 g_x$ for all $x \in K$. 
\item Boundedness: There are constants $c_1, c_2 > 0$ such that $L(x, v) \ge c_1|v| - c_2$ for all $(x, v) \in TM$. 
\end{enumerate}
By \eqref{item:strict-convexity}, $L$ is uniformly strictly convex in $v$ uniformly on compact subsets of $M$ and so $\lim_{|v| \to \infty} \frac{L(x, v)}{|v|} = \infty$ for every $x \in M$ uniformly on compact subsets of $M$. Thus, in particular, $L$ is a Tonelli-Lagrangian (see, e.g., \cite{Fathi}). 

In fact, one could more generally also consider explicitly time dependent Lagrangians $L : [a, b] \times TM \to \R$. Under a suitable additional assumption (see below) the necessary modifications are straightforward so that we will not dwell on this point. On the other hand, Assumption (iii) is too restrictive for certain applications. For this reason we will in Section \ref{section:Unbounded-energy} also investigate systems which instead of (ii) and (iii) satisfy 
\begin{itemize}
\item[(ii')] Uniform strict convexity: There exists some positive constant $c_0$ such that $\nabla^2_v L(x, \cdot) \ge c_0 g_x$ for all $x \in M$. 
\item[(iii')] Growth condition: There are constants $c_1, c_2 > 0$ and a reference point $x_0 \in M$ such that $L(x, v) \ge c_1|v|^2 - c_2(1 + d_M^2(x, x_0))$ for all $(x, v) \in TM$, where $d_M$ denotes the (geodesic) distance on $M$. 
\end{itemize}
Throughout Sections \ref{section:Lagrangian-systems} and \ref{section:Limits}, however, Assumptions (i), (ii) and (iii) are assumed to hold. 

Fix an interval $[a, b] \subset \R$. To $\gamma \in C^{ac}([a, b]; M)$, i.e., absolutely continuous curves $\gamma : [a, b] \to M$, we associate the action 
$$ {\cal A}(\gamma) := \int_a^b L(\gamma(t), \dot{\gamma}(t)) \, dt. $$ 
By Assumption (iii) this integral exists in $(-\infty, +\infty]$. As it will be convenient in the sequel, we extend ${\cal A}$ to $C([a, b]; M)$ simply by setting ${\cal A}(\gamma) = + \infty$ if $\gamma : [a, b] \to M$ is continuous but not absolutely continuous. 

Critical points of this action functional with respect to variations that keep the endpoints fixed are called extremals. In particular, for every $x, y \in M$ there exists a minimizing extremal connecting $x$ and $y$, i.e., a minimizer of ${\cal A}$ among all absolutely continuous curves $\gamma : [a, b] \to M$ with $\gamma(a) = x$ and $\gamma(b) = y$. The value of the action of such a minimizing extremal will be denoted $c(x, y)$. Then $c : M \times M \to \R$ is continuous. By our hypotheses, minimizing extremals $\gamma$ are $C^{\infty}$ and satisfy the Euler-Lagrange equation 
$$ \frac{d}{dt} (\nabla_v L)(\gamma(t), \dot{\gamma}(t)) = (\nabla_x L)(\gamma(t), \dot{\gamma}(t)). $$
In fact, there exists a smooth flow $\phi^L_t$ on $TM$ such that every minimizing extremal $\gamma$ satisfies $(\gamma(t), \dot{\gamma}(t)) = \phi^L_{t-s}(\gamma(s), \dot{\gamma}(s))$ for $s, t \in [a, b]$, i.e., $\gamma$ lies on an orbit of $\phi^L_t$. 

\begin{rem}  
For Lagrangians explicitly depending on time, the existence of such a flow, which action minimizers follow, has to be assumed.
\end{rem}

Although we will state our results below in terms of the Euler-Lagrange flow $\phi^L_t$, we recall here that the dynamical evolution can alternatively be described through the associated Hamiltonian and its flow. The Legendre transform in $v$ of the Lagrangian $L$ defines the Hamiltonian $H \in C^{\infty}(T^*M; \R)$, 
$$ H(x, p) := \sup_{v \in T_x M} \left( p \cdot v - L(x, v) \right). $$
By strict convexity, the right hand side is maximized at $p = \nabla_v L(x, v)$ and one defines the global Legendre transform as the diffeomorphism $\nabla_v L : TM \to T^*M$, $(x, v) \mapsto (x, \nabla_v L(x, v))$. Via the canonical symplectic structure on $T^*M$, one associates a vector field $X_H$, which is uniquely determined by 
$$ X_H(x, p) = (\nabla_p H(x, p), - \nabla_x H(x, p)) $$ 
in local coordinates for $M$. The smooth flow associated to $X_H$ is denoted $\phi^H_t$. The global Legendre transform conjugates $\phi^L$ and $\phi^H$: 
$$ \phi^L_t 
   = (\nabla_v L)^{-1} \circ \phi^H_t \circ \nabla_v L. $$

\begin{ex}\label{ex:kin-pot} 
The most prominent example which we will also discuss from a numerical point of view in Section \ref{section:Numerics} is given by 
$$ L(x, v) = \frac{m}{2} |v|^2 - V(x) $$ 
for some constant $m > 0$ and potential $V \in C^{\infty}(\R^n; \R)$, which is bounded on $M = \R^n$. (See Section \ref{section:Unbounded-energy} for unbounded $V$.) The corresponding Euler-Lagrange equation is Newton's equation of motion 
$$ m \ddot{\gamma}(t) = - \nabla V(\gamma(t)). $$ 
\end{ex}

\subsection{Many particle systems}\label{subsection:Lagrangian-systems-many-body}

We proceed to give an elementary account on Lagrangian systems with finitely many particles. Suppose $L$ is a Lagrangian on the manifold $M$ as above. We consider a system of $N$ identical particles moving along curves $t \mapsto \gamma_i(t)$ on $M$ with initial positions $\gamma_i(a) = x_i$ (with $x_i \ne x_j$ for $i \ne j$) at time $t = a$ and final positions $\gamma_i(b) = y_i$, $i = 1, \ldots, N$. The associated Lagrangian action is 
$$ {\cal A}_N(\gamma) 
   := \frac{1}{N} \sum_{i = 1}^N {\cal A}(\gamma_i) $$ 
acting on continuous curves $\gamma = (\gamma_1, \ldots, \gamma_N) \in C([a, b]; M^N)$ and so, in particular, 
$$ {\cal A}_N(\gamma) 
   = \frac{1}{N} \sum_{i = 1}^N \int_a^b L(\gamma_i(t), \dot{\gamma}_i(t)) \, dt $$ 
for absolutely continuous curves $\gamma = (\gamma_1, \ldots, \gamma_N) \in C^{ac}([a, b]; M^N)$. The factor $\frac{1}{N}$ is introduced so as to measure the action per particle. It gives the right scaling in our convergence analysis to be discussed below. For the analysis of the discrete action functional with fixed $N$ it is of course irrelevant. 

Being interested in the behavior of the particle density, we define the probability measures 
\begin{align}\label{eq:mu-N}
  \mu^{(N)}_a := \frac{1}{N} \sum_{i = 1}^N \delta_{x_i} 
  \qquad\mbox{and}\qquad 
  \mu^{(N)}_b := \frac{1}{N} \sum_{i = 1}^N \delta_{y_i} 
\end{align}
and -- assuming that the particles are indistinguishable -- are led to minimizing the action functional ${\cal A}_N(\gamma)$ among all $\gamma \in C([a, b]; M^N)$ such that 
$$ \frac{1}{N} \sum_{i = 1}^N \delta_{\gamma_i(a)}  
   = \mu^{(N)}_a 
   \qquad\mbox{and}\qquad 
   \frac{1}{N} \sum_{i = 1}^N \delta_{\gamma_i(b)} 
   = \mu^{(N)}_b. $$
This in turn amounts to minimizing $\frac{1}{N} \sum_{i = 1}^N {\cal A}(\gamma_i)$ over the set 
\begin{align*}
  \Big\{ \gamma \in C([a, b]; M^N) : \,
  &\gamma_i(a) = x_i,~ \gamma_i(b) = T(x_i)~ \forall \, i \\ 
  &\quad\text{ for some } T \text{ with } T \# \mu_a^{(N)} = \mu_b^{(N)} \Big\}, 
\end{align*}
where $T \# \mu$ denotes the push forward of a measure $\mu$ under the mapping $T$. 

By first fixing $T$, we see that the optimal $\gamma_i$ need to be minimizing extremals connecting $x_i$ and $T(x_i)$. Moreover, the minimal value of the action is 
\begin{align*}
  &\min \left\{ \frac{1}{N} \sum_{i = 1}^N c(x_i, T(x_i)) 
  : T \# \mu_a^{(N)} = \mu_b^{(N)} \right\} \\ 
  &= \min \left\{ \int_M c(x, T(x)) \, \mu_a^{(N)}(dx) 
  : T \# \mu_a^{(N)} = \mu_b^{(N)} \right\}, 
\end{align*} 
where $c(x, y)$ is the minimal action of curves connecting $x$ and $y$ as defined above. This shows that, corresponding to our initial and final distributions, $T$ is a solution the Monge-Kantorovich mass transportation problem with Lagrangian cost $c$. Of course, this is a finite problem for $T$ and so an optimal $T$ always exists, showing that with $\gamma_i$ a minimizing extremal connecting $x_i$ and $T(x_i)$, $\gamma$ is in fact a minimizer for ${\cal A}_N$. 

If $\gamma = (\gamma_1, \ldots, \gamma_N)$ is a minimizer and so $\gamma_i$ a minimizing extremal for every $i$, we set $v_i := \dot{\gamma}_i(a)$. This way we can write 
$$ (\gamma_i(t), \dot{\gamma}_i(t)) 
   = \phi^L_{t-a}(x_i, v_i) $$ 
and thus 
$$ \gamma_i 
   = \pi_M \phi^L_{\cdot - a}(x_i, v_i). $$ 
Here $\pi_M \phi^L_{\cdot - a} : TM \to C([a, b]; M)$ is the (continuous) Euler-Lagrange flow mapping 
$$ (x, v) \mapsto (t \mapsto \pi_M \phi^L_{t-a}(x, v)). $$

\section{Convergence to the infinite particle system}\label{section:Limits}

In this section we first prove our main convergence result on the passage to a continuum system in the limit of infinitely many particles. We then analyze the limiting problem by deriving analogous properties to the finite system with the help of our convergence scheme. Finally we investigate convergence properties of stationary points.

\subsection{Variational convergence to a continuum system}\label{section:Limits-Gamma}

In order to pass to a continuum system, we consider the general action functional 
$$ \mathbb{A}(\pi) 
   = \int {\cal A}(\gamma) \, d\pi(\gamma) $$ 
acting on the space of Borel probability measures ${\cal P}(C([a, b]; M))$. Here the space of curves $C([a, b]; M)$ is endowed with the metric of uniform convergence 
$$ d_{\Gamma}(\gamma, \gamma') 
   = \sup_{t \in [a, b]} d_M(\gamma(t), \gamma'(t)), $$ 
where $d_M$ denotes the (geodesic) distance on $M$. $(C([a, b]; M), d_{\Gamma})$ then is a Polish space. 

Note that with the help of Lemmas \ref{lemma:A-coercive} and \ref{lemma:A-lsc} below the direct method of the calculus of variations can be applied to show existence of minimizers on weakly closed sets of probability measures. In this section we aim at proving a general $\Gamma$-convergence result which in addition shows that any minimizer can be approximated by almost minimizers of finite systems with pre-assigned initial and final distributions. This not only is the starting point for our numerical convergence analysis in Section \ref{section:Numerics}, but will also allow for an investigation of the minimizers of $\mathbb{A}$ by recourse to an analysis of the minimizers of ${\cal A}_N$. 

For any initial and final distribution $\mu_a$ and $\mu_b$, respectively, we define the restricted functionals $\mathbb{A}_{\mu_a, \mu_b}$ by 
$$ \mathbb{A}_{\mu_a, \mu_b}(\pi) 
   = \begin{cases} \mathbb{A}(\pi), & \mbox{if } \pr_a \# \pi = \mu_a \mbox{ and } \pr_b \# \pi = \mu_b, \\ 
     \infty & \mbox{otherwise.} 
     \end{cases} $$ 
Here $\pr_t : C([a, b]; M) \to M$ denotes the evaluation at time $t$: $\pr_t(\gamma) = \gamma(t)$. 

\begin{rem}  
If $\mu_a^{(N)}$ and $\mu_b^{(N)}$ are as in \eqref{eq:mu-N}, then any $\pi$ with $\mathbb{A}_{\mu_a^{(N)}, \mu_b^{(N)}}(\pi) < \infty$ satisfies $\pr_a \pi = \mu^{(N)}_a = \frac{1}{N} \sum_{i = 1}^N \delta_{x_i}$ with $x_i \ne x_j$ for $i \ne j$. The proof of Theorem \ref{theo:Gamma-convergence}(\ref{enum:limsup}) will show that all the results in this paper remain true if we impose the additional constraint that $\mathbb{A}_{\mu_a^{(N)}, \mu_b^{(N)}}(\pi) = \infty$ unless $\pi$ is of the form  
$$ \pi = \frac{1}{N} \sum_{i = 1}^N \delta_{\gamma_i} $$
for suitable curves $\gamma_i$ with $\gamma_i(a) = x_i$, so that any such $\pi$ satisfies 
$$ \mathbb{A}_{\mu_a^{(N)}, \mu_b^{(N)}}(\pi) = {\cal A}_N(\gamma_1, \ldots, \gamma_N). $$ 
\end{rem}

\subsubsection*{Statement of the main results}

Fix two compactly supported absolutely continuous probability measures $\mu_a$ and $\mu_b$ on $M$. Let $\mu_a^{(N)}, \mu_b^{(N)}$ be as in \eqref{eq:mu-N} such that there exists a compact set in $M$ supporting all these discrete measures. 

\begin{theo}\label{theo:Gamma-convergence} 
If $\mu_a^{(N)} \weakly \mu_a$ and $\mu_b^{(N)} \weakly \mu_b$, then $\mathbb{A}_{\mu_a^{(N)}, \mu_b^{(N)}}$ $\Gamma$-converges to $\mathbb{A}_{\mu_a, \mu_b}$ on ${\cal P}(C([a, b]; M))$ with respect to weak convergence, i.e.:
\begin{enumerate}[(i)]
\item\label{enum:liminf} liminf-inequality: Whenever $\pi^{(N)} \weakly \pi$ in ${\cal P}(C([a, b]; M))$, then 
$$ \liminf_{N \to \infty} \mathbb{A}_{\mu_a^{(N)}, \mu_b^{(N)}}(\pi^{(N)}) 
   \ge \mathbb{A}_{\mu_a, \mu_b}(\pi) $$ 
and 
\item\label{enum:limsup} recovery sequences: for any $\pi \in {\cal P}(C([a, b]; M))$ there exists a sequence $\pi^{(N)} \weakly \pi$ in ${\cal P}(C([a, b]; M))$ such that 
$$ \lim_{N \to \infty} \mathbb{A}_{\mu_a^{(N)}, \mu_b^{(N)}}(\pi^{(N)}) 
   = \mathbb{A}_{\mu_a, \mu_b}(\pi). $$ 
Moreover, $\pi^{(N)}$ can be chosen as $\pi^{(N)} = \frac{1}{N} \sum_{i = 1}^N \delta_{\gamma_i^{(N)}}$ for suitable curves $\gamma^{(N)}_i \in C^{ac}([a, b]; M)$. 
\end{enumerate} 
\end{theo}

In view of our numerical convergence results in Section \ref{section:Numerics} for optimal transportation meshfree methods as developed in \cite{LiHabbalOrtiz}, we note that our assumptions in Theorem \ref{theo:Gamma-convergence} allow for maximal flexibility in sampling approximating marginal measures: With the exception of requiring a common compact support, the only assumption is that $\mu_a^{(N)} \weakly \mu_a$ and $\mu_b^{(N)} \weakly \mu_b$. In particular, it is not necessary to assume that these discrete measures are derived from the limiting continuum measures by some particular coarse graining procedure. 

This $\Gamma$-convergence theorem is complemented by the following compactness result. 

\begin{theo}\label{theo:compactness}
If $\pi^{(N)}$ is a sequence of probability measures such that the measures $\pr_a \# \pi^{(N)}$ are supported on a common compact set and the corresponding sequence $\mathbb{A}(\pi^{(N)})$ of actions is bounded, then $\pi^{(N)}$ is relatively compact with respect to weak convergence. 
\end{theo}

A standard argument in the theory of $\Gamma$-convergence now implies the convergence of minimizers.
\begin{cor}\label{cor:Minimizers-converge}
Under the assumptions of Theorem \ref{theo:Gamma-convergence}, 
$$ \lim_{N \to \infty} \min_{\pi} \mathbb{A}_{\mu_a^{(N)}, \mu_b^{(N)}}(\pi) 
   = \min_{\pi} \mathbb{A}_{\mu_a, \mu_b}(\pi). $$ 
Moreover, if $\pi^{(N)}$ is a sequence of almost minimizers of $\mathbb{A}_{\mu_a^{(N)}, \mu_b^{(N)}}$, i.e., 
$$ \lim_{N \to \infty} \mathbb{A}_{\mu_a^{(N)}, \mu_b^{(N)}}(\pi^{(N)}) = \min_{\pi}\mathbb{A}_{\mu_a, \mu_b}(\pi), $$ 
then $\pi^{(N)}$ admits a weakly convergent subsequence. In fact, any weakly convergent subsequence converges weakly to a minimizer of $\mathbb{A}_{\mu_a, \mu_b}$.
\end{cor}

\subsubsection*{Proofs of Theorems \ref{theo:Gamma-convergence} and \ref{theo:compactness}}

In the proof of Theorem \ref{theo:Gamma-convergence} it will be advantageous to use a metric that metrizes weak convergence on ${\cal P}(C([a, b]; M))$. We therefore recall that weak convergence of probability measures (as in any Polish space) is equivalent to convergence in the bounded Lipschitz distance 
$$ d_{\rm BL}(\mu, \nu) 
   := \sup\left\{ \int \varphi \, d(\mu - \nu) : \| \varphi \|_{\rm Lip} \le 1 \right\}, $$ 
where for any Lipschitz continuous $\varphi : C([a, b]; M) \to \R$, 
$$ \| \varphi \|_{\rm Lip} 
   := \sup_{\gamma} |\varphi(\gamma)| + \sup_{\gamma \ne \tilde{\gamma}} \frac{|\varphi(\gamma) - \varphi(\tilde{\gamma})|}{d_{\Gamma}(\gamma, \tilde{\gamma})}. $$ 

We start by proving two (well known) preparatory lemmas on coercivity and lower semicontinuity.
\begin{lemma}\label{lemma:A-coercive}
For every compact set $K \subset M$ and any $C > 0$ the set 
$$ \{ \gamma \in C([a, b]; M) : \gamma(a) \in K,~ {\cal A}(\gamma) \le C\} $$ 
is relatively compact in $C([a, b]; M)$. 
\end{lemma}
Note that this notion of coerciveness is stronger than the mild coercivity assumption asking that the set of minimizing extremals between two compact sets in $M$ be compact in $C([a, b]; M)$. As this Lemma is crucial in our compactness results, we include the straightforward proof. 

\begin{proof} 
If $\gamma \in C([a, b]; M)$ satisfies ${\cal A}(\gamma) \le C$, then it is absolutely continuous with 
\begin{align*}
  C 
  \ge \int_a^b L(\gamma(t), \dot{\gamma}(t)) \, dt 
  \ge c_1 \int_a^b |\dot{\gamma}(t)| \, dt - c_2(b - a). 
\end{align*}
But then  (with the convention of denoting generic constants of different value by the same symbol $C$) 
$$ d_M(\gamma(t), \gamma(a)) \le C. $$
So by the Hopf-Rinow theorem there is a compact set $K' \subset M$ such that $\gamma(t) \in K'$ for all $t \in [a, b]$, whenever $\gamma(a) \in K$ and ${\cal A}(\gamma) \le C$.  

By our assumption on $L$ there are constants $c, c' > 0$ such that 
$$ L(x, v) \ge c|v|^2 - c' \qquad \forall \, x \in K',~ v \in T_x M. $$ 
So ${\cal A}(\gamma) \le C$ implies that for $a \le t_1 < t_2 \le b$ 
\begin{align*}
  d_M(\gamma(t_1), \gamma(t_2)) 
  \le \int_{t_1}^{t_2} |\dot{\gamma}(t)| \, dt 
  \le \left( \int_{t_1}^{t_2} |\dot{\gamma}(t)|^2 \, dt  \right)^{\frac{1}{2}} \left( t_2 - t_1 \right)^{\frac{1}{2}} 
  \le C \sqrt{t_2 - t_1}. 
\end{align*}

We have thus shown that $\{\gamma \in C([a, b]; M) : \gamma(a) \in K,~ {\cal A}(\gamma) \le C\}$ is pointwise compact and equicontinuous. The Arzela-Ascoli theorem then yields that this set is relatively compact. 
\end{proof}

\begin{lemma}\label{lemma:A-lsc} 
${\cal A}$ is lower semicontinuous on $(C([a, b]; M), d_{\Gamma})$. 
\end{lemma} 

\begin{proof} 
As $\gamma_k \to \gamma$ implies that $\gamma_k|_{[t_1, t_2]} \to \gamma|_{[t_1, t_2]}$ uniformly on every subinterval $[t_1, t_2]$ with $a \le t_1 < t_2 \le b$, by splitting the action integral we may without loss of generality assume that all $\gamma_k$ and $\gamma$ are covered by a single coordinate chart. We may also assume that ${\cal A}(\gamma_k)$ is bounded. Then in local coordinates, 
\begin{align*}
  C 
  \ge \int_a^b L(\gamma(t), \dot{\gamma}(t)) \, dt 
  \ge c \int_a^b |\dot{\gamma}(t)|^2 \, dt - c' 
\end{align*}
for some constants $c, c' > 0$, so that $\gamma_k$ is bounded uniformly in $W^{1,2}((a, b); \R^n)$, where $n = \dim M$. But then $\gamma_k \weakly \gamma$ in $W^{1, 2}((a, b); \R^n)$ and by convexity of $L$ in $v$ we obtain 
\[
\liminf_{k \to \infty} \int_a^b L(\gamma_k(t), \dot{\gamma}_k(t)) \, dt 
   \ge \int_a^b L(\gamma(t), \dot{\gamma}(t)) \, dt. \qedhere \]
\end{proof}

To abbreviate we write $\mathbb{A}$ for $\mathbb{A}_{\mu_a, \mu_b}$ and $\mathbb{A}_N$ for $\mathbb{A}_{\mu_a^{(N)}, \mu_b^{(N)}}$ in the following proof. 
\begin{proof}[Proof of Theorem \ref{theo:Gamma-convergence}]  
\eqref{enum:liminf} In order to prove the $\liminf$-inequality, we assume that $\pi^{(N)} \weakly \pi$ and, without loss of generality, $\mathbb{A}_N(\pi^{(N)}) < \infty$, so that, in particular, $\pr_t \# \pi^{(N)} = \mu_t^{(N)}$ for $t \in \{a, b\}$ and all $N \in \N$. Thus, also $\pr_t \# \pi = \mu_t$ for $t \in \{a, b\}$. As ${\cal A}$ is lower semicontinuous by Lemma \ref{lemma:A-lsc}, the claim now follows from the portmanteau theorem: 
$$ \liminf_{N \to \infty} \int {\cal A}(\gamma) \, d\pi^{(N)}(\gamma) 
   \ge \int {\cal A}(\gamma) \, d\pi(\gamma). $$ 

\eqref{enum:limsup} It remains to construct a recovery sequence for given $\pi \in {\cal P}(C([a, b]; M))$. Note that in view of \eqref{enum:liminf}, it suffices to show that $\limsup_{N \to \infty} \mathbb{A}_{N}(\pi^{(N)}) \le \mathbb{A}(\pi)$. In fact, a diagonal sequence argument shows that it is enough to prove that, for any $\eps > 0$ there exists a sequence $\pi^{(N)}$ with 
$$ \limsup_{N \to \infty} d_{\rm BL}(\pi^{(N)}, \pi) \le O(\eps) 
   \quad\mbox{and}\quad 
   \limsup_{N \to \infty} \mathbb{A}_{N}(\pi^{(N)}) \le \mathbb{A}(\pi) + O(\eps). $$

Let $\eps > 0$ and choose a compact set $K \subset C([a, b]; M)$ with $\pi(K) \ge 1 - \eps$. Then choose disjoint measurable sets $V_1, \ldots, V_m$ of diameter at most $\eps$ covering $K$ such that $\mu_t(\partial \pr_t V_i) = 0$ for $i = 1, \ldots, m$ and $t \in \{a, b\}$. (Such sets can be constructed by first covering $K$ with $m$ balls $B_{\eps/4}(\tilde{\gamma}_i)$ of radius $\frac{\eps}{4}$, then choosing $\eps_i \in (\frac{\eps}{4}, \frac{\eps}{2})$ such that $\mu_a(\partial B_{\eps_i}(\tilde{\gamma}_i(a))) = \mu_b(\partial B_{\eps_i}(\tilde{\gamma}_i(b))) = 0$ and setting $V_1 = B_{\eps_1}(\tilde{\gamma}_1)$, $V_i = B_{\eps_i}(\tilde{\gamma}_i) \setminus (V_1 \cup \ldots \cup V_{i-1})$ for $i \ge 2$.) If necessary splitting the sets $V_i$, we may furthermore assume that, for $t \in \{a, b\}$, 
$$ \pr_t(V_i) \cap \pr_t(V_j) = \emptyset 
   \quad\mbox{or}\quad 
   \pr_t(V_i) = \pr_t(V_j) $$ 
for all $i, j$. Now choose $\gamma_i \in V_i$ with ${\cal A}(\gamma_i) < \inf_{\gamma \in V_i}{\cal A}(\gamma) + \eps$.  

As $\mu_t(\partial \pr_t V_i) = 0$ for $t \in \{a, b\}$, we have 
$$ \lim_{N \to \infty} \mu_t^{(N)}(\pr_t(V_i)) 
   = \mu_t(\pr_t(V_i)) $$ 
for $t \in \{a, b\}$, i.e., the number of points in $\{x_1, \ldots, x_N\}$ that lie in $\pr_a(V_i)$ as well as the number of points in $\{y_1, \ldots, y_N\}$ (counted with multiplicities) that lie in $\pr_b(V_i)$ is 
$$ N \mu_t^{(N)}(\pr_t(V_i)) 
   \ge N \!\!\sum_{j : \pr_t(V_j) \atop \quad = \pr_t(V_i)} \pi(V_j) + o(N) $$ 
for $t = a$ or $t = b$, respectively. So it is possible to find disjoint sets $W_i \subset \{1, \ldots, x_N\}$ and a mapping $T : \bigcup_{i = 1}^{m} W_i \to \{y_1, \ldots, y_N\}$, which is injective if $\{y_1, \ldots, y_N\}$ is viewed as a multiset, such that 
$$ W_i \subset \pr_a(V_i),\quad 
   T(W_i) \subset \pr_b(V_i) \quad\mbox{and}\quad 
   \# W_i = N \pi(V_i) + o(N). $$ 

Now for each $x \in W_i$ choose a curve $\gamma_x$ with ${\cal A}(\gamma_x) < {\cal A}(\gamma_i) + C\eps$ by setting 
$$ \gamma_x(t) = \gamma_i\left( \frac{(b-a)t - (a+b)\eps}{b-a-2\eps}\right) $$ 
for $t \in [a + \eps, b - \eps]$ and connecting $x$ to $\gamma_i(a)$ on $[a, a+\eps]$ and $\gamma_i(b)$ to $T(x)$ on $[b-\eps, b]$ on geodesics with velocities bounded uniformly in $\eps$. For $x \notin \bigcup_{i = 1}^m W_i$ choose $T(x) \in \{y_1, \ldots, y_N\}$ such that $T$ becomes a bijection when $\{y_1, \ldots, y_N\}$ is viewed as a multiset and let $\gamma_x$ be any curve connecting $x$ to $T(x)$ such that ${\cal A}(\gamma_x) \le C$ for some constant $C$. (This is possible since by assumption all initial and end points are contained in a common compact set.)

We define $\pi^{(N)}$ by 
$$ \pi^{(N)} = \frac{1}{N} \sum_{i = 1}^N \delta_{\gamma_{x_i}}. $$
Clearly, $\pr_t \#\pi^{(N)} = \mu_t^{(N)}$ for $t \in \{a, b\}$. In order to calculate $d_{\rm BL}(\pi^{(N)}, \pi)$, let $\varphi : C([a, b]; M) \to \R$ Lipschitz with $\| \varphi \|_{\rm Lip} \le 1$. Then, with $V_0 :=  C([a, b]; M) \setminus \bigcup_{i = 1}^m V_i$ and $W_0 = \{x_1, \ldots, x_N\} \setminus \bigcup_{i = 1}^m W_i$, 
\begin{align*}
  \int \varphi \, d(\pi^{(N)} - \pi) 
  &= \sum_{i = 0}^m \left( \frac{1}{N} \sum_{x \in W_i} \varphi(\gamma_x) - \int_{V_i} \varphi \, d\pi \right). 
\end{align*}
Now for $i = 0$ and sufficiently large $N$ we have 
\begin{align*}
  \left| \frac{1}{N} \sum_{x \in W_i} \varphi(\gamma_x) - \int_{V_i} \varphi \, d\pi \right| 
  \le \frac{\# W_0}{N} + \pi(V_0) 
  \le 2\pi(V_0) + \eps \le 3 \eps. 
\end{align*}
For $i \ge 1$ we obtain from $d_{\Gamma}(\gamma, \tilde{\gamma}) \le C \eps$ for any $\gamma, \tilde{\gamma} \in \{ \gamma_x : x \in W_i\} \cup V_i$ 
\begin{align*}
  \left| \frac{1}{N} \sum_{x \in W_i} \varphi(\gamma_x) - \int_{V_i} \varphi \, d\pi \right| 
  &\le \left| \frac{\# W_i}{N} \varphi(\gamma_i) - \pi(V_i) \varphi(\gamma_i) \right| 
    + C \eps \left( \frac{\#W_i}{N} + \pi(V_i) \right) \\
  &\le C \eps 
\end{align*}
for large $N$. Summing over all $i$ we have shown that indeed 
\begin{align*}
  \limsup_{N \to \infty} d_{\rm BL}(\pi^{(N)}, \pi) 
   \le C \eps. 
\end{align*} 

Finally we have to estimate the value of the action of $\pi^{(N)}$. 
\begin{align*}
  \mathbb{A}_N(\pi^{(N)}) 
  &= \int {\cal A}(\gamma) \, d \pi^{(N)} 
  = \sum_{i = 0}^m \frac{1}{N} \sum_{x \in W_i} {\cal A}(\gamma_x) \\ 
  &\le \frac{1}{N} \sum_{x \in W_0} {\cal A}(\gamma_x) 
    + \sum_{i = 1}^m \frac{1}{N} \sum_{x \in W_i} \left( {\cal A}(\gamma_i) + C \eps \right) \\ 
  &\le C \frac{\# W_0}{N} + \sum_{i = 1}^m \frac{\# W_i}{N} \left( \inf_{\gamma \in V_i} {\cal A}(\gamma) + C\eps \right) \\ 
  &\le o(1) + O(\eps) + \sum_{i = 1}^m (1 + o(1)) \pi(V_i)  \inf_{\gamma \in V_i} {\cal A}(\gamma) \\ 
  &\le \int {\cal A}(\gamma) \, d\pi + o(1) + O(\eps) 
  = \mathbb{A}(\pi) + o(1) + O(\eps). 
\end{align*}
This concludes the proof. 
\end{proof} 

We proceed to prove Theorem \ref{theo:compactness}. 

\begin{proof}[Proof of Theorem \ref{theo:compactness}]
Suppose $\mathbb{A}(\pi^{(N)}) = \int {\cal A}(\gamma) \, d \pi^{(N)} \le C$. By Prohorov's theorem it suffices to show that the sequence $\pi^{(N)}$ is tight. Let $\eps > 0$. As by Assumption (iii) ${\cal A}$ is bounded from below by $-c_2(b - a)$, we have 
$$ \pi^{(N)} \left( \left\{ \gamma \in C([a, b]; M) : {\cal A}(\gamma) \ge \frac{C + c_2(b-a)}{\eps} \right\} \right) 
   \le \eps, $$ 
for otherwise 
$$ \int {\cal A}(\gamma) \, d \pi^{(N)} 
   > \frac{C + c_2(b-a)}{\eps} \cdot \eps - c_2(b - a) 
   = C. $$
Since by assumption there is a compact set $K \subset M$ such that $\pi^{(N)}$-a.e.\ curve $\gamma$ satisfies $\gamma(a) \in K$, we deduce from Lemma \ref{lemma:A-coercive} that there exists a compact set $K_{\eps} \in C([a, b]; M)$ with 
$$ \pi^{(N)}(K_{\eps}) \ge 1 - \eps $$ 
for every $N$. This concludes the proof. 
\end{proof}

\subsection{Analysis of the limiting continuum system}\label{section:Limits-Analysis}

In the language of the theory of optimal transport, minimizers of $\mathbb{A}_{\mu_a,\mu_b}$ are dynamical optimal transference plans from $\mu_a$ to $\mu_b$ with respect to the Lagrangian cost $c$. These objects have been investigated intensively over the last years, see, e.g., \cite{Villani:09} for a recent account and, in particular, \cite{BernardBuffoni:07} for a result on compact manifolds analogous to Theorem \ref{theo:flow} below. 

For $K_1, K_2 \subset M$ we denote by $\Gamma^{\rm min}_{K_1, K_2}$ the set of minimizing extremals starting in $K_1$ and ending in $K_2$. Our main result on the behavior of minimizers of $\mathbb{A}$ is the following. 
\begin{theo}\label{theo:flow}
If $\pi$ is a minimizer for $\mathbb{A}_{\mu_a, \mu_b}$, then 
\begin{enumerate}[(i)] 
\item\label{item:pi-supp} $\pi$ is supported on $\Gamma^{\rm min}_{K_a, K_b}$ for $K_a = \supp(\mu_a)$, $K_b = \supp(\mu_b)$ and  

\item\label{item:flow} there exists a measure $\eta$ on $TM$ such that $\pi = \pi_M \phi^L_{\cdot - a} \# \eta$. 
\end{enumerate}
\end{theo} 

We begin with the following (well known) preparation: 
\begin{lemma}\label{lemma:Gamma-min-cpt}
If $K_1, K_2 \subset M$ are compact, then $\Gamma^{\rm min}_{K_1, K_2}$ is compact in $C([a, b]; M)$. 
\end{lemma} 

\begin{proof} 
As the continuous function $c$ is bounded on $K_1 \times K_2$, by Lemma \ref{lemma:A-coercive} it suffices to show that $\Gamma^{\rm min}_{K_1, K_2}$ is closed. Suppose $\Gamma^{\rm min}_{K_1, K_2} \ni \gamma_k \to \gamma$. Then $\gamma(a) \in K_1$ and $\gamma(b) \in K_2$. If $\gamma$ were not a minimizing extremal, there would be $\tilde{\gamma}$ with $\tilde{\gamma}(a) = \gamma(a)$, $\tilde{\gamma}(b) = \gamma(b)$ and, by Lemma \ref{lemma:A-lsc}, 
$$ \liminf_{k \to \infty} {\cal A}(\gamma_k) \ge {\cal A}(\gamma) > {\cal A}(\tilde{\gamma}). $$
But then there exists an $\eps > 0$ such that for all $k$ sufficiently large one can construct a curve $\gamma_{k,\eps}$ by setting 
$$ \gamma_{k,\eps}(t) = \tilde{\gamma}\left( \frac{(b-a)t - (a+b)\eps}{b-a-2\eps}\right) $$ 
for $t \in [a + \eps, b - \eps]$ and connecting $\gamma_k(a)$ to $\tilde{\gamma}(a)$ on $[a, a+\eps]$ and $\tilde{\gamma}(b)$ to $\gamma_k(b)$ on $[b-\eps, b]$ suitably, so that ${\cal A}(\gamma_{k, \eps}) \le \frac{{\cal A}(\tilde{\gamma}) + {\cal A}(\gamma)}{2}$, contradicting the minimality of $\gamma_k$. 
\end{proof}  

\begin{lemma}\label{lemma:min-extr-quant}
Let $K_1, K_2 \subset M$ be compact. Then there exists $\omega : [0, \infty) \to \R$ continuous with $\omega(0) = 0$ and $\omega(s) > 0$ for $s > 0$ such that 
$$ {\cal A}(\gamma) \ge c(x, y) + \omega(\dist(\gamma, \Gamma^{\rm min}_{K_1, K_2})) $$ 
for every $\gamma \in C([a, b]; M)$ with $\gamma(a) = x \in K_1$ and $\gamma(b) = y \in K_2$. 
\end{lemma}

\begin{proof} 
It suffices to show that for every $\eps > 0$ there exists $\delta > 0$ such that 
$$ \dist(\gamma, \Gamma^{\rm min}_{K_1, K_2}) \ge \eps \Longrightarrow {\cal A}(\gamma) - c(x, y) \ge \delta $$
for every $\gamma \in C([a, b]; M)$ with $\gamma(a) = x \in K_1$ and $\gamma(b) = y \in K_2$. 

If this were not the case, then there would be an $\eps > 0$ and sequence $\gamma^{(k)}$ with $\gamma^{(k)}(a) = x_k \in K_1$, $\gamma^{(k)}(b) = y_k \in K_2$, $\dist(\gamma^{(k)}, \Gamma^{\rm min}_{K_1, K_2}) \ge \eps$ and ${\cal A}(\gamma^{(k)}) - c(x_k, y_k) \to 0$. $c$ is continuous and, in particular, bounded on $K_1 \times K_2$. So ${\cal A}(\gamma^{(k)})$ is bounded and, by Lemma \ref{lemma:A-coercive}, thus has a convergent subsequence (not relabeled) $\gamma^{(k)} \to \gamma$ with $\gamma(a) = x := \lim_{k\to\infty} x_k \in K_1$ and $\gamma(b) = y := \lim_{k\to\infty} y_k \in K_2$. 

Now the lower semicontinuity of ${\cal A}$ implies that 
$$ {\cal A}(\gamma) 
   \le \liminf_{k \to \infty} {\cal A}(\gamma^{(k)}) 
   = \lim_{k \to \infty} c(x_k, y_k) 
   = c(x, y), $$ 
which shows that $\gamma$ is a minimizing extremal connecting $x$ and $y$. This contradicts the fact that 
\[ \dist(\gamma, \Gamma^{\rm min}_{K_1, K_2}) 
   = \lim_{k \to \infty} \dist(\gamma^{(k)}, \Gamma^{\rm min}_{K_1, K_2}) 
   \ge \eps. \qedhere \] 
\end{proof}  

\begin{proof}[Proof of Theorem \ref{theo:flow}(i)]
We again write $\mathbb{A}$ for $\mathbb{A}_{\mu_a, \mu_b}$ and $\mathbb{A}_N$ for $\mathbb{A}_{\mu_a^{(N)}, \mu_b^{(N)}}$. 
Let $\pi$ be a minimizer of $\mathbb{A}$. For each $N \in \N$ choose point sets $\{x_1, \ldots, x_N\} \subset K_a$ and $\{y_1, \ldots, y_N\} \subset K_b$ such that 
$$ \mu_a^{(N)} = \frac{1}{N} \sum_{i = 1}^N \delta_{x_i} \weakly \mu_a 
   \quad \mbox{and} \quad 
   \mu_b^{(N)} = \frac{1}{N} \sum_{i = 1}^N \delta_{y_i} \weakly \mu_b. $$ 
Suppose $\pi^{(N)} = \frac{1}{N} \sum_{i = 1}^N \delta_{\gamma_i}$ with $\gamma_i(a) = x_i$ and $\gamma_i(b) = T(x_i)$ is a recovery sequence for $\pi$ with respect to the $\Gamma$-convergence of $\mathbb{A}_N$ to $\mathbb{A}$. With $\omega$ as in Lemma \ref{lemma:min-extr-quant} for $K_1 = K_a$ and $K_2 = K_b$ we obtain (cf.\ Section \ref{subsection:Lagrangian-systems-many-body} for the second inequality)  
\begin{align*} 
  \mathbb{A}(\pi) 
  &= \lim_{N \to \infty} \mathbb{A}_N(\pi^{(N)}) 
   = \lim_{N \to \infty} \frac{1}{N} \sum_{i = 1}^N {\cal A}(\gamma_i) \\ 
  &\ge \frac{1}{N} \limsup_{N \to \infty} \sum_{i = 1}^N \big( c(x_i, T(x_i)) 
     + \omega(\dist(\gamma_i, \Gamma^{\rm min}_{K_a, K_b})) \big) \\ 
  &\ge \lim_{N \to \infty} \min_{\tilde{\pi}} \mathbb{A}_N(\tilde{\pi}) 
     + \lim_{N \to \infty} \int \omega(\dist(\gamma, \Gamma^{\rm min}_{K_a, K_b})) \, d \pi^{(N)}(\gamma) \\ 
  &= \min_{\tilde{\pi}} \mathbb{A}(\tilde{\pi}) + \int \omega(\dist(\gamma, \Gamma^{\rm min}_{K_a, K_b})) \, d \pi(\gamma). 
\end{align*}
Thus, $\dist(\gamma, \Gamma^{\rm min}_{K_a, K_b}) = 0$ $\pi$-a.e.\ and the claim follows. 
\end{proof}

For the proof of Theorem \ref{theo:flow}\eqref{item:flow} we will need some finer estimates on recovery sequences. 
\begin{lemma}\label{lemma:bounded-v}
For any two compact sets $K_1, K_2 \subset M$ there exists a constant $C > 0$ such that $|\dot{\gamma}(t)| \le C$ for all $\gamma \in \Gamma^{\rm min}_{K_1, K_2}$ and $t \in [a, b]$.
\end{lemma} 

We include the short proof of this well known estimate. 
\begin{proof} 
For every $\gamma \in \Gamma^{\rm min}_{K_1, K_2}$ it follows from 
$$ \int_a^b L(\gamma(t), \dot{\gamma}(t)) \, dt \le \max_{(x, y) \in K_1 \times K_2} c(x, y) \le C $$
and the lower boundedness of $L$ that there is some $t_{\gamma}$ with $|\dot{\gamma}(t_{\gamma})| \le C$. But then $(\gamma(t), \dot{\gamma}(t)) = \phi^L_{t - t_{\gamma}}(\gamma(t_{\gamma}), \dot{\gamma}(t_{\gamma}))$ is bounded uniformly in $t$ and $\gamma$ since $(\gamma(t), \dot{\gamma}(t))$ solves the Euler-Lagrange equation. 
\end{proof}

\begin{lemma}\label{lemma:min-extr-approx}
Suppose $\pi$ is a minimizer of $\mathbb{A}_{\mu_a, \mu_b}$ and let $K_a = \supp(\mu_a)$, $K_b = \supp(\mu_b)$. Then for every $\eps > 0$ there exists a probability measure $\pi_{\eps} = \frac{1}{N} \sum_{i = 1}^N \delta_{\gamma_i} \in {\cal P}(C([a, b]; M))$ supported on $\Gamma^{\rm min}_{K_a, K_b}$ such that 
$$ d_{\rm BL}(\pi_{\eps}, \pi) \le \eps. $$
\end{lemma}

\begin{proof} 
Choose a recovery sequence $\pi^{(N)}$ precisely as in the proof of Theorem \ref{theo:flow}\eqref{item:pi-supp}, so that 
\begin{align*} 
  \lim_{N \to \infty} \int \omega(\dist(\gamma, \Gamma^{\rm min}_{K_a, K_b})) \, d \pi^{(N)}(\gamma) 
  = 0. 
\end{align*}
Setting $G = \{\gamma : \dist(\gamma, \Gamma^{\rm min}_{K_a, K_b}) \le \eps\}$, for $N$ sufficiently large we thus have $\pi^{(N)} (G) \ge 1 - \eps$. In particular, for every $\gamma_i \in \supp (\pi^{(N)}) \cap G$ there exists $\tilde{\gamma}_i \in \Gamma^{\rm min}_{K_a, K_b}$ with $d_{\Gamma}(\gamma_i, \tilde{\gamma}_i) \le \eps$. For $\gamma_i \in \supp (\pi^{(N)}) \setminus G$ we choose $\tilde{\gamma}_i \in \Gamma^{\rm min}_{K_a, K_b}$ arbitrarily. Accordingly we define $\tilde{\pi}^{(N)} = \frac{1}{N} \sum_{i = 1}^N \delta_{\tilde{\gamma}_i}$.

Now for any Lipschitz $\varphi : C([a, b]; M) \to \R$ with $\|\varphi\|_{\rm Lip} \le 1$, 
\begin{align*}
  \int \varphi \, d(\tilde{\pi}^{(N)} - \pi^{(N)}) 
  &\le \frac{1}{N}\sum_{i : \gamma_i \in G} |\varphi(\gamma_i) - \varphi(\tilde{\gamma}_i)| 
    + \frac{1}{N}\sum_{i : \gamma_i \notin G} \left( |\varphi(\gamma_i)| + |\varphi(\tilde{\gamma}_i)| \right) \\ 
  &\le \eps + \frac{2 \#\{i : \gamma_i \notin G\}}{N} 
   \le 3 \eps. 
\end{align*}
This shows that $d_{\rm BL} (\tilde{\pi}^{(N)}, \pi^{(N)}) \le 3\eps$. But $\pi^{(N)}$ is a recovery sequence for $\pi$, whence also $d_{\rm BL} (\pi^{(N)}, \pi) \le \eps$ for large $N$. It follows that $d_{\rm BL} (\tilde{\pi}^{(N)}, \pi) \le 4\eps$. 
\end{proof}

\begin{proof}[Proof of Theorem \ref{theo:flow}(ii)] 
Let $\pi$ be a minimizer of $\mathbb{A}_{\mu_a, \mu_b}$. By Lemma \ref{lemma:min-extr-approx} there are measures $\pi^{(N)} = \frac{1}{N} \sum_{i = 1}^N \delta_{\gamma_i}$ with $\gamma_i \in \Gamma^{\rm min}_{K_a, K_b}$ and such that $\pi^{(N)} \weakly \pi$. With $x_i = \gamma_i(a)$ and $v_i := \dot{\gamma}_i(a)$ we can write 
$$ \gamma_i = \phi^L_{\cdot - a}(x_i, v_i). $$ 
So the measures 
$$ \eta^{(N)} 
   = \frac{1}{N} \sum_{i = 1}^N \delta_{(x_i, v_i)} $$
on $TM$ satisfy $\pi^{(N)} = \pi_M \phi^L_{\cdot - a} \# \eta^{(N)}$. Note that $|v_i| \le C$ for some constant $C$ independent of $N$ by Lemma \ref{lemma:bounded-v}. Being supported on a common compact subset of $TM$, there exists a subsequence (not relabeled) such that $\eta^{(N)} \weakly \eta$. But then also $\pi_M \phi^L_{\cdot - a} \# \eta^{(N)} \weakly \pi_M \phi^L_{\cdot - a} \# \eta$ and hence $\pi = \pi_M \phi^L_{\cdot - a} \# \eta$. 
\end{proof}

\subsection{Convergence of stationary points}\label{section:Limits-Stat-points}

In this section we investigate the limiting behavior of critical points which do not necessarily need to be minimizers. In a sense to be made precise, we will show that in the discrete-to-continuum limit extremals converge to a distribution in phase space which follows the Euler-Lagrange flow. We consider sequences of curves $\gamma^{(N)} = (\gamma^{(N)}_1, \ldots, \gamma^{(N)}_N) \in C([a, b]; M^N)$ and their empirical measures $\pi^{(N)} = \frac{1}{N} \sum_{i = 1}^N \delta_{\gamma_i^{(N)}}$ with uniformly bounded values of the action ${\cal A}_N(\gamma^{(N)})$ which are stationary with respect to variations in $\gamma^{(N)}$. ${\cal A}_N$ being the sum of single particle contributions, for every $i$ the curve $\gamma_i$ then satisfies the Euler-Lagrange equation
$$ \frac{d}{dt} (\nabla_v L)(\gamma^{(N)}_i(t), \dot{\gamma}^{(N)}_i(t)) = (\nabla_x L)(\gamma^{(N)}_i(t), \dot{\gamma}^{(N)}_i(t)). $$

The following theorem gives a precise version in which sense extremal curves converge to a measure supported on extremals that evolves according to the Euler-Lagrange flow. For two compact sets $K_1, K_2 \subset M$ let $\Gamma_{K_1, K_2}$ be the the set of extremals $\gamma$ starting in $K_1$ and ending in $K_2$ and set $\Gamma_{K_1, K_2}^C = \{ \gamma \in \Gamma_{K_1, K_2} : {\cal A}(\gamma) \le C \}$. 

\begin{theo}\label{theo:stationary-convergence} 
Suppose $\gamma^{(N)} \in C([a, b]; M^N)$ is a sequence of stationary points for ${\cal A}_N$ with ${\cal A}_N(\gamma^{(N)}) \le C$ for some constant $C$ and $\gamma^{(N)}(a) \in K_a^N$, $\gamma^{(N)}(b) \in K_b^N$ for compact sets $K_a, K_b \subset M$. Let $\pi^{(N)} =  \frac{1}{N} \sum_{i = 1}^N \delta_{\gamma_i^{(N)}}$. Then for a subsequence (not relabeled) $\pi^{(N)} \weakly \pi$ for some $\pi \in {\cal P}(C([a,b]; M))$. $\pi$ is supported on $\Gamma_{K_a, K_b}$ and there is a measure $\eta$ on $TM$ such that $\pi = \pi_M \phi^L_{\cdot - a} \# \eta$. 
\end{theo}

We begin by proving an extension of Lemmas \ref{lemma:Gamma-min-cpt} and \ref{lemma:bounded-v}. 

\begin{lemma}\label{lemma:stationary-prep}
Suppose $K_1, K_2 \subset M$ are compact. There is a compact set $K \subset M$ and a constant $c = c(K_1, C)$ such that for all $\gamma \in \Gamma_{K_1, K_2}^C$ 
$$ \gamma(t) \in K \quad \text{and} \quad |\dot{\gamma}(t)| \le c \quad \text{for all} \quad t \in [a, b]. $$ 
Furthermore, $\Gamma_{K_1, K_2}^C$ is closed. 
\end{lemma}

\begin{proof}
By Lemma \ref{lemma:A-coercive} there exists a compact set $K \subset M$ with $\gamma(t) \in K$ for all $\gamma \in \Gamma_{K_1, K_2}^C$ and $t \in [a, b]$. The second claim, namely, $|\dot{\gamma}| \le c$, follows exactly as in the proof of Lemma \ref{lemma:bounded-v}.

Now if $\Gamma_{K_1, K_2}^C \ni \gamma_k \to \gamma$, then on the one hand we have 
$$ {\cal A}(\gamma) \le \liminf_{k \to \infty} {\cal A}(\gamma_k) \le C $$ by Lemma \ref{lemma:A-lsc}. On the other hand, $\gamma$ is indeed an extremal: By splitting, if necessary, the action functional we may without loss of generality assume that all $\gamma_k$ and $\gamma$ are covered by a single coordinate chart. Then in local coordinates by the boundedness of $|\dot{\gamma}_k|$, in fact also $|\ddot{\gamma}_k|$ is bounded uniformly in $k$ since each $\gamma_k$ solves the Euler-Lagrange equation and $L$ is strictly convex on compacts by Assumption (ii). We therefore have $\gamma_k \stackrel{\ast}{\weakly} \gamma$ in $W^{2, \infty}((a, b); \R^n)$. But then we may pass to the limit in the Euler-Lagrange equations for $\gamma_k$ to obtain that also $\gamma$ solves the Euler-Lagrange equation. 
\end{proof}

\begin{proof}[Proof of Theorem \ref{theo:stationary-convergence}] 
By Theorem \ref{theo:compactness} there is a subsequence (not relabeled) such that $\pi^{(N)} \weakly \pi$ for some $\pi \in {\cal P}(C([a,b]; M))$. 

Let $\eps > 0$. As in the proof of Theorem \ref{theo:compactness} we find 
\begin{align}\label{eq:piNA-Absch}
  \pi^{(N)} \left( \left\{ \gamma \in C([a, b]; M) : {\cal A}(\gamma) \ge c_{\eps} \right\} \right) 
  \le \eps, 
\end{align}
for $c_{\eps} = \frac{C + c_2(b-a)}{\eps}$ and so $\pi^{(N)}(\Gamma_{K_a, K_b}^{c_{\eps}}) \ge 1 - \eps$. Since $\Gamma_{K_a, K_b}^{c_{\eps}}$ by Lemma \ref{lemma:stationary-prep} is closed, we deduce from $\pi^{(N)} \weakly \pi$ and the portmanteau theorem that 
$$ \pi(\Gamma_{K_1,K_2}) 
   \ge \pi(\Gamma_{K_1,K_2}^{c_{\eps}}) 
   \ge \limsup_{N \to \infty} \pi^{(N)}(\Gamma_{K_1,K_2}^{c_{\eps}}) 
   \ge 1 - \eps. $$
As $\eps$ was arbitrary, $\pi(\Gamma_{K_1,K_2}) = 1$. 

Similarly as in the proof of Theorem \ref{theo:flow}\eqref{item:flow} with $v_i := \dot{\gamma}_i(a)$ we can write 
$$ \gamma_i = \phi^L_{\cdot - a}(x_i, v_i) $$ 
for $\gamma^{(N)} = (\gamma_1, \ldots, \gamma_N)$. So the measures 
$$ \eta^{(N)} 
   = \frac{1}{N} \sum_{i = 1}^N \delta_{(x_i, v_i)} $$
on $TM$ satisfy $\pi^{(N)} = \pi_M \phi^L_{\cdot - a} \# \eta^{(N)}$. By Lemma \ref{lemma:stationary-prep} there exists $c = c(\eps) > 0$ such that $|v_i| \le c$ for all $i$ with $\gamma_i \in \Gamma_{K_1,K_2}^{c_{\eps}}$, so that 
$$ \eta^{(N)}(\{(x, v) \in TM : x \in K_a,~ |v| \le c\}) 
   \ge \pi^{(N)}(\Gamma_{K_a, K_b}^{c_{\eps}}) \ge 1 - \eps $$ 
by \eqref{eq:piNA-Absch}. Thus being tight, there exists a subsequence (not relabeled) such that $\eta^{(N)} \weakly \eta$. But then also $\pi_M \phi^L_{\cdot - a} \# \eta^{(N)} \weakly \pi_M \phi^L_{\cdot - a} \# \eta$ and hence $\pi = \pi_M \phi^L_{\cdot - a} \# \eta$. 
\end{proof}

\section{Systems with unbounded potential energy}\label{section:Unbounded-energy}

The theory set forth so far accounts for Lagrangians that are bounded from below, as required by Assumption (iii) in Section \ref{section:Lagrangian-systems}. For this reason, we had to assume that the potential energy $V$ in the example of page \pageref{ex:kin-pot} is bounded (at least from above). This is unsatisfactory from an application's point of view as not even a simple harmonic oscillator satisfies such an assumption. In this section we will show that, under generic assumptions on the Lagrangian $L$, our $\Gamma$-convergence and compactness results extend to systems whose potential energy is not necessarily bounded from above, if the time span $b - a$ over which trajectories are observed is sufficiently small. Such an assumption is in fact necessary as otherwise the action is not bounded from below. The convergence of stationary points, however, can still be justified on long time spans under reasonable assumptions.  

In the sequel we will assume that $L : TM \to \R$ satisfies 
\begin{itemize}
\item[(i)] Smoothness: $L \in C^{\infty}(TM; \R)$, 
\item[(ii')] Uniform strict convexity: There exists some positive constant $c_0$ such that $\nabla^2_v L(x, \cdot) \ge c_0 g_x$ for all $x \in M$. 
\item[(iii')] Growth condition: There are constants $c_1, c_2 > 0$ such that $L(x, v) \ge c_1|v|^2 - c_2(1 + d_M^2(x, x_0))$ for all $(x, v) \in TM$ and a reference point $x_0 \in M$. 
\end{itemize}

\subsubsection*{Statement of the main results}

The action ${\cal A}(\gamma)$ is defined as before. (Note that by (iii') for $\gamma \in C^{ac}([a, b]; M)$ the integral $\int_a^b L(\gamma(t), \dot{\gamma}(t)) \, dt$ exists in $(-\infty, + \infty]$.) Also the many particle action ${\cal A}_N$ and the general action functional $\mathbb{A}$ as well as the initial and final measures $\mu_a^{(N)}$, $\mu_a$ and $\mu_b^{(N)}$, $\mu_b$, respectively, are defined as in Section \ref{section:Limits}. 

\begin{theo}\label{theo:Gamma-convergence-Eukl} 
Let $b - a \le \sqrt{\frac{c_1}{4 c_2}}$. With the notation of Theorem \ref{theo:Gamma-convergence}, if $\mu_a^{(N)} \weakly \mu_a$ and $\mu_b^{(N)} \weakly \mu_b$, then $\mathbb{A}_{\mu_a^{(N)}, \mu_b^{(N)}}$ $\Gamma$-converges to $\mathbb{A}_{\mu_a, \mu_b}$ on ${\cal P}(C([a, b]; M))$ with respect to weak convergence. 
\end{theo}

Again we also have a compactness result: 
\begin{theo}\label{theo:compactness-Eukl}
Let $b - a \le \sqrt{\frac{c_1}{4 c_2}}$. If $\pi^{(N)}$ is a sequence of probability measures such that the measures $\pr_a \# \pi^{(N)}$ are supported on a common compact set and the corresponding sequence $\mathbb{A}(\pi^{(N)})$ of actions is bounded, then $\pi^{(N)}$ is relatively compact with respect to weak convergence. 
\end{theo}

As a conclusion to Theorems \ref{theo:Gamma-convergence-Eukl} and \ref{theo:compactness-Eukl} we remark that, under the additional assumption that $b - a \le \sqrt{\frac{c_1}{4 c_2}}$, Corollary \ref{cor:Minimizers-converge} applies verbatim to the systems considered in this section. Also Theorems \ref{theo:flow} and \ref{theo:stationary-convergence} remain valid for these systems if $b - a \le \sqrt{\frac{c_1}{4 c_2}}$. This follows in a straightforward manner from the arguments in Section \ref{section:Limits} using that Lemma \ref{lemma:bounded-traj} below bounds trajectories by their value of the action functional. (Note that the constant $\sqrt{\frac{c_1}{4 c_2}}$ is not sharp.)

It is well known that compactness may be lost on long time intervals, rendering the search for action minimizers meaningless: 

\begin{ex}\label{ex:cpt-loss} For $M = \R$, $L(x, v) = \frac{m}{2} v^2 - \frac{c}{2} x^2$ with $c > m > 0$ and $a = 0$, $b = \sqrt{\frac{m}{c}} \pi$, for every $\alpha \in \R$ the curve $\gamma(t) = \alpha \sin(\sqrt{\frac{c}{m}} t)$ is an extremal connecting $0$ to itself with ${\cal A}(\gamma) = \frac{\alpha^2}{4}(\sqrt{mc} - c)$, which diverges to $-\infty$ as $\alpha \to \infty$. 
\end{ex} 

On the other hand, under suitable local bounds on the action, we may still formulate a version of Theorem \ref{theo:stationary-convergence} for unbounded potentials on long time intervals if the Euler-Lagrange flow $\phi^L_t$ is complete. 

\begin{theo}\label{theo:stationary-convergence-Eukl-long} 
Suppose that the Euler-Lagrange flow $\phi^L_t$ is complete. Let $\gamma^{(N)} \in C([a, b]; M^N)$ be a sequence of $N$-tuples of solutions of the Euler-Lagrange equation with $\gamma^{(N)}(a) \in K_a^N$, $\gamma^{(N)}(b) \in K_b^N$ for compact sets $K_a, K_b \subset M$. Assume that there is a constant $C > 0$ such that 
$$ {\cal A}_N(\gamma^{(N)}|_{[a, b']}) 
   = \frac{1}{N} \sum_{i = 1}^N \int_{a}^{b'} L(\gamma^{(N)}_i(t), \dot{\gamma}^{(N)}_i(t)) \, dt 
   \le C $$
for some $b' > a$ with $b' \le \min\{b, \sqrt{\frac{c_1}{4 c_2}}\}$. Then $\pi^{(N)} =  \frac{1}{N} \sum_{i = 1}^N \delta_{\gamma_i^{(N)}}$ admits a subsequence (not relabeled) such that $\pi^{(N)} \weakly \pi$ for some $\pi \in {\cal P}(C([a,b]; M))$. $\pi$ is supported on $\Gamma_{K_a, K_b}$ and there is a measure $\eta$ on $TM$ such that $\pi = \pi_M \phi^L_{\cdot - a} \# \eta$. 
\end{theo}

\subsubsection*{Proofs of Theorems \ref{theo:Gamma-convergence-Eukl}, \ref{theo:compactness-Eukl} and \ref{theo:stationary-convergence-Eukl-long}}

We begin by estimating intermediate particle positions by the action. 
\begin{lemma}\label{lemma:bounded-traj}
Let $b - a \le \sqrt{\frac{c_1}{4 c_2}}$. For every compact set $K \subset M$ and any $C > 0$ there exists a constant $c = c(c_1, c_2, K, C)$ such that for all $\gamma \in C^{ac}([a, b]; M)$ with $\gamma(a) \in K$ and ${\cal A}(\gamma) \le C$ 
$$ d_M(\gamma(t), x_0) \le c \quad \forall \, t \in [a, b]. $$
\end{lemma}

\begin{proof}
As 
\begin{align*} 
  \int_a^b d_M^2(\gamma(t), x_0) \, dt 
  &\le \int_a^b 2 d_M^2(\gamma(a), x_0) + 2 d_M^2(\gamma(a), \gamma(t)) \, dt \\ 
  &\le 2 (b - a) d_M^2(\gamma(a), x_0) + 2 (b - a) \left( \int_a^b |\dot{\gamma}(s)| \, ds \right)^2 \\ 
  &\le 2 (b - a) d_M^2(\gamma(a), x_0) + 2 (b - a)^2 \int_a^b |\dot{\gamma}(s)|^2 \, ds, 
\end{align*}
by Assumption (iii') on $L$ we have 
\begin{align*}
  {\cal A}(\gamma) 
  &\ge c_1 \int_a^b |\dot{\gamma}(t)|^2 \, dt - c_2 (b - a) - c_2 \int_a^b d_M^2(\gamma(t), x_0) \, dt \\ 
  &\ge \left( c_1 - 2 c_2 (b - a)^2 \right) \int_a^b |\dot{\gamma}(t)|^2 \, dt - c_2 (b - a) (1 + 2 d_M^2(\gamma(a), x_0)) \\ 
  &\ge \frac{c_1}{2}\int_a^b |\dot{\gamma}(t)|^2 \, dt - \sqrt{c_1 c_2} (1 + d_M^2(\gamma(a), x_0)). 
\end{align*}  
But then for every $t \in [a, b]$ 
\begin{align*}
  d_M(\gamma(t), \gamma(a)) 
  &\le \int_a^b |\dot{\gamma}(t)| \, dt 
  \le \sqrt{b - a} \left( \int_a^b |\dot{\gamma}(t)|^2 \, dt \right)^{\frac{1}{2}} \\ 
  &\le \left( \frac{2 (b - a)}{c_1} {\cal A}(\gamma) + 2 \sqrt{\frac{c_2}{c_1}} (1 + d_M^2(\gamma(a), x_0)) (b - a) \right)^{\frac{1}{2}}, 
\end{align*}
and so $d_M(\gamma(t), x_0) \le d_M(\gamma(a), x_0) + d_M(\gamma(a), \gamma(t)) \le c$ for $c = c(c_1, c_2, K, C)$ sufficiently large. 
\end{proof}

\begin{proof}[Proof of Theorem \ref{theo:Gamma-convergence-Eukl}]
(i) For $R > 0$ let $\theta_R \in C^{\infty}(M; \R_+)$ with $\theta_R(x) = 0$ for $d_M(x, x_0) \le R$ and $\theta_R(x) \ge c_2(1 + |x|^2)$ for $d_M(x, x_0) \ge 2R$ and define the Lagrangian $L^{R}(x, v) = L(x, v) + \theta_R(x)$. Then $L^R$ satisfies Assumptions (i), (ii) and (iii) from Section \ref{section:Lagrangian-systems}. Denote the induced action functionals by ${\cal A}^R$, ${\cal A}_N^R$, $\mathbb{A}^R_{\mu_a, \mu_b}$ and $\mathbb{A}^R_{\mu_a^{(N)}, \mu_b^{(N)}}$, respectively. 

For given $k \in \N$, by Lemma \ref{lemma:bounded-traj} we may choose $R$ so large that for each $\gamma \in C([a, b]; M)$ with $\gamma(a) \in \overline{\bigcup_{N \in \N} \supp\, \mu_a^{(N)}}$ we have 
$$ {\cal A}(\gamma) \le k \implies d_M(\gamma(t), x_0) \le c \le R \implies {\cal A}^R(\gamma) = {\cal A}(\gamma). $$
But then ${\cal A}(\gamma) \ge \min\{k, {\cal A}^R(\gamma)\}$ for every such $\gamma$. 

Now assume that $\pi^{(N)} \weakly \pi$ in ${\cal P}(C([a, b]; M))$ with uniformly bounded $\mathbb{A}^R_{\mu_a^{(N)}, \mu_b^{(N)}}(\pi^{(N)})$. Then as in the proof of Theorem \ref{theo:Gamma-convergence}, $\pr_t \# \pi = \mu_t$ for $t \in \{a, b\}$. By Lemma \ref{lemma:A-lsc} ${\cal A}^R$ is lower semicontinuous and so is $\gamma \mapsto \min\{k, {\cal A}^R(\gamma)\}$, so the portmanteau theorem gives  
\begin{align*}
  \liminf_{N \to \infty} \int {\cal A}(\gamma) \, d\pi^{(N)}(\gamma) 
  &\ge \liminf_{N \to \infty} \int \min\{k, {\cal A}^R(\gamma)\} \, d\pi^{(N)}(\gamma) \\ 
  &\ge \int \min\{k, {\cal A}^R(\gamma)\} \, d\pi(\gamma) \\ 
  &\ge \int \min\{k, {\cal A}(\gamma)\} \, d\pi(\gamma). 
\end{align*}
As $k$ was arbitrary, by the monotone convergence theorem we obtain 
\begin{align*}
  \liminf_{N \to \infty} \int {\cal A}(\gamma) \, d\pi^{(N)}(\gamma) 
  \ge \int {\cal A}(\gamma) \, d\pi(\gamma). 
\end{align*}
in the limit $k \to \infty$. 

(ii) For the construction of a recovery sequence of a given measure $\pi \in {\cal P}(C([a, b]; M))$ it suffices to note that all the measures $\pi^{(N)}$ of the recovery sequence obtained in the proof of Theorem \ref{theo:Gamma-convergence} are supported on curves with values in a common compact set, so that with $\mathbb{A}^R$ as above for $R$ sufficiently large 
\begin{align*}
  \limsup_{N \to \infty} \mathbb{A}_{\mu_a^{(N)}, \mu_b^{(N)}}(\pi^{(N)}) 
  &= \limsup_{N \to \infty} \mathbb{A}_{\mu_a^{(N)}, \mu_b^{(N)}}^R(\pi^{(N)}) \\ 
  &\le \mathbb{A}_{\mu_a, \mu_b}^R(\pi) + O(\eps) 
   = \mathbb{A}_{\mu_a, \mu_b}(\pi) + O(\eps) 
\end{align*}
with arbitrary $\eps > 0$ as in the proof of Theorem \ref{theo:Gamma-convergence}. 
\end{proof}

\begin{proof}[Proof of Theorem \ref{theo:compactness-Eukl}] 
This follows along the lines of the proof of Theorem \ref{theo:compactness}: As ${\cal A}(\gamma) \le C$ by Lemma \ref{lemma:bounded-traj} implies that $d_M(\gamma(t), x_0) \le c$ for $\pi^{(N)}$-a.e.\ $\gamma$, ${\cal A}$ by Assumption (iii') is bounded from below on the joint support of the $\pi^{(N)}$. As in the proof of Theorem \ref{theo:compactness} this implies that for given $\eps > 0$ there is a constant $\tilde{C} > 0$ such that 
$$ \pi^{(N)} \left( \left\{ \gamma \in C([a, b]; M) : {\cal A}(\gamma) \ge \frac{\tilde{C}}{\eps} \right\} \right) 
   \le \eps. $$ 
Now for $\pi^{(N)}$-a.e.\ $\gamma$ we have $d_M(\gamma(t), x_0) \le c(\eps)$ if ${\cal A}(\gamma) \le \frac{\tilde{C}}{\eps}$ again by Lemma \ref{lemma:bounded-traj}. So with ${\cal A}^R$ as in the proof of Theorem \ref{theo:Gamma-convergence-Eukl} for $R$ sufficiently large ${\cal A}(\gamma) = {\cal A}^R(\gamma)$ and by Lemma \ref{lemma:A-coercive} there is a compact set $K_{\eps}$ with 
$$ \left\{ \gamma \in C([a, b]; M) : {\cal A}(\gamma) \le \frac{\tilde{C}}{\eps} \right\} 
   \subset \left\{ \gamma \in C([a, b]; M) : {\cal A}^R(\gamma) \le \frac{\tilde{C}}{\eps} \right\} 
   \subset K_{\eps} $$ 
and therefore $\pi^{(N)}(K_{\eps}) \ge \pi^{(N)}(\{ \gamma \in C([a, b]; M) : {\cal A}(\gamma) \le \frac{\tilde{C}}{\eps}\}) \ge 1 - \eps$. The sequence $\pi^{(N)}$ is thus tight. 
\end{proof} 

\begin{proof}[Proof of Theorem \ref{theo:stationary-convergence-Eukl-long}] 
Invoking Lemma \ref{lemma:bounded-traj} we may argue as in the proof of Theorems \ref{theo:stationary-convergence} and \ref{theo:compactness-Eukl} to find some $c_{\eps} > 0$ with 
\begin{align}\label{eq:piNA-Absch-Eukl}
  \pi^{(N)} \left( \left\{ \gamma \in C([a, b]; M) : {\cal A}(\gamma|_{[a, b']}) \ge c_{\eps} \right\} \right) 
  \le \eps. 
\end{align}
Similarly as in the proof of Theorem \ref{theo:flow}\eqref{item:flow} setting $v_i := \dot{\gamma}_i(a)$ so that 
$$ \gamma_i = \phi^L_{\cdot - a}(x_i, v_i) $$ 
for $\gamma^{(N)} = (\gamma_1, \ldots, \gamma_N)$, the measures 
$$ \eta^{(N)} 
   = \frac{1}{N} \sum_{i = 1}^N \delta_{(x_i, v_i)} $$
on $TM$ satisfy $\pi^{(N)} = \pi_M \phi^L_{\cdot - a} \# \eta^{(N)}$. 

Applying Lemmas \ref{lemma:stationary-prep} and \ref{lemma:bounded-traj} on the time interval $[a, b']$ we find $c = c(\eps) > 0$ such that $|v_i| \le c$ for all $i$ with ${\cal A}(\gamma_i) \le c_{\eps}$, so that 
$$ \eta^{(N)}(\{(x, v) \in TM : x \in K_a,~ |v| \le c\}) 
   \ge 1 - \eps $$ 
by \eqref{eq:piNA-Absch-Eukl}. Thus being tight, there exists a subsequence (not relabeled) such that $\eta^{(N)} \weakly \eta$. But then also $\pi_M \phi^L_{\cdot - a} \# \eta^{(N)} \weakly \pi_M \phi^L_{\cdot - a} \# \eta$ as measures on $C([a, b]; M)$ and hence $\pi = \pi_M \phi^L_{\cdot - a} \# \eta$. 
\end{proof} 

\section{Numerical schemes}\label{section:Numerics}

In this section we show how our results can be applied to investigate convergence properties of optimal transportation meshfree methods as described in \cite{LiHabbalOrtiz}. In such a scheme also the time is discretized, and we need to extend our analysis to a set-up where simultaneously the time step converges to $0$ while the number of particles tends to infinity. To this end, by identifying the time discretized system as a perturbation of the systems analyzed in Sections \ref{section:Limits} and \ref{section:Unbounded-energy} we will be able to reduce to our earlier results. While such a reduction is possible for many numerical quadrature schemes, by way of example we concentrate on the midpoint rule here. For the sake of simplicity we also restrict our analysis to Lagrangians $L : \R^n \times \R^n \to \R$ of the form  
$$ L(x, v) = \frac{m}{2} |v|^2 - V(x) $$ 
with $m > 0$, where $V \in C^{\infty}(\R^n)$ satisfies $|V(x)| \le c_2(1 + |x|^2)$ for some $c_2 > 0$. Note that by conservation of the energy $\frac{m}{2} |\dot{\gamma}(t)|^2 + V(\gamma(t))$ for solutions $\gamma$ of the Euler-Lagrange equation, in particular one has $|\dot{\gamma}(t)| \le C(1 + |\dot{\gamma}(t)|)$ so that by Gronwall's inequality the associated Euler-Lagrange flow is complete. 

For a given triangulation ${\cal T}_h = \{(\tau_{j - 1}, \tau_j) : j = 1, \ldots, l\}$ of the time interval $(a, b)$, where $a = \tau_0 < \tau_1 < \ldots < \tau_l = b$ with $|\tau_j - \tau_{j - 1}| \le h$, the corresponding action ${\cal A}$ shall be discretized by the midpoint rule as 
$$ {\cal A}^{[h]}(\gamma) 
   = \sum_{j = 1}^{l} \frac{m}{2} \frac{|\gamma(\tau_j) - \gamma(\tau_{j-1})|^2}{\tau_j - \tau_{j-1}} - V \left( \frac{\gamma(\tau_j) + \gamma(\tau_{j-1})}{2} \right) (\tau_j - \tau_{j - 1}). $$ 
In fact, ${\cal A}^{[h]}$ is a variational integrator only depending on the nodal values $\gamma(\tau_j)$. As it is convenient, we may and will assume that ${\cal A}^{[h]}(\gamma)$ is finite only for $\gamma$ piecewise affine subordinate to ${\cal T}_h$, in which case we may also write 
$$ {\cal A}^{[h]}(\gamma) 
   = \int_a^b \frac{m}{2} |\dot{\gamma}(t)|^2 \, dt - \sum_{j = 1}^{l} V \left( \gamma \left(\frac{\tau_j + \tau_{j-1}}{2} \right) \right) (\tau_j - \tau_{j - 1}). $$ 

\subsubsection*{Statement of the main results}

Suppose $\mu_a, \mu_b, \mu_a^{(N)}, \mu_b^{(N)}$ are as in Section \ref{section:Limits} and let $\mathbb{A}^{[h]}$, $\mathbb{A}_{\mu_a, \mu_b}^{[h]}$ and $\mathbb{A}^{[h]}_{\mu_a^{(N)}, \mu_b^{(N)}}$ denote the time discretized action functionals on ${\cal P}(C([a, b]; \R^n))$ obtained by replacing ${\cal A}$ by ${\cal A}^{[h]}$. 
\begin{theo}\label{theo:Gamma-convergence-Numerics} 
Suppose $V$ is bounded or $b - a \le \sqrt{\frac{m}{32 c_2}}$. Let $h_N > 0$ with $h_N \to 0$ as $N \to \infty$. With the notation of Theorem \ref{theo:Gamma-convergence}, if $\mu_a^{(N)} \weakly \mu_a$ and $\mu_b^{(N)} \weakly \mu_b$, then $\mathbb{A}^{[h_N]}_{\mu_a^{(N)}, \mu_b^{(N)}}$ $\Gamma$-converges to $\mathbb{A}_{\mu_a, \mu_b}$ on ${\cal P}(C([a, b]; \R^n))$ with respect to weak convergence. 
\end{theo}

Again we also have a compactness result: 
\begin{theo}\label{theo:compactness-Numerics} 
Suppose $V$ is bounded or $b - a \le \sqrt{\frac{m}{32 c_2}}$. Let $h_N > 0$ with $h_N \to 0$ as $N \to \infty$. If $\pi^{(N)}$ is a sequence of probability measures such that the measures $\pr_a \# \pi^{(N)}$ are supported on a common compact set and the corresponding sequence $\mathbb{A}^{[h_N]}(\pi^{(N)})$ of actions is bounded, then $\pi^{(N)}$ is relatively compact with respect to weak convergence. 
\end{theo}

Again we note that, as a conclusion to Theorems \ref{theo:Gamma-convergence-Numerics} and \ref{theo:compactness-Numerics}, if $V$ is bounded or $b - a \le \sqrt{\frac{m}{32 c_2}}$, then Corollary \ref{cor:Minimizers-converge} remains true if $\mathbb{A}_{\mu_a^{(N)}, \mu_b^{(N)}}$ is replaced by $\mathbb{A}^{[h_N]}_{\mu_a^{(N)}, \mu_b^{(N)}}$. For a result on stationary points corresponding to Theorem \ref{theo:stationary-convergence}, however, we do have to argue differently, as the discretized trajectories only solve a discretized version of the Euler-Lagrange equations and, in particular, are not extremals of ${\cal A}$. (We also remark that the constant $\sqrt{\frac{m}{32 c_2}}$ is not sharp.) 

\begin{theo}\label{theo:stationary-convergence-Numerics} 
Let $\gamma^{(N)} \in C([a, b]; (\R^n)^N)$ be a sequence of stationary points for ${\cal A}^{[h_N]}_N$ with $\gamma^{(N)}(a) \in K_a^N$ for a compact set $K_a \subset \R^n$, where $h_N > 0$ with $h_N \to 0$ as $N \to \infty$. Assume that 
$$ {\cal A}^{[h_N]}_N(\gamma^{(N)}|_{[a, b']}) 
   \le C $$ 
for some constant $C$ and $b' > a$ with $b' - a \le \sqrt{\frac{m}{32 c_2}}$. Then there exists a subsequence of $\pi^{(N)} =  \frac{1}{N} \sum_{i = 1}^N \delta_{\gamma_i^{(N)}}$ (not relabeled) such that $\pi^{(N)} \weakly \pi$ for some $\pi \in {\cal P}(C([a,b]; \R^n))$. $\pi$ is supported on $\Gamma_{K_a, K_b}$ and there is a measure $\eta$ on $\R^{2n}$ such that $\pi = \pi_{\R^n} \phi^L_{\cdot - a} \# \eta$. 
\end{theo}

\subsubsection*{Proofs of Theorems \ref{theo:Gamma-convergence-Numerics}, \ref{theo:compactness-Numerics} and \ref{theo:stationary-convergence-Numerics}}

Similarly as in Lemma \ref{lemma:bounded-traj} we have:
\begin{lemma}\label{lemma:bounded-traj-neu}
If $V$ is bounded or if $b - a \le \sqrt{\frac{m}{32 c_2}}$, then there exists a constant $c = c(m, c_2, R)$ such that for all $\gamma \in C^{ac}([a, b]; \R^n)$ with $|\gamma(a)| \le R$  
$$ \int_a^b |\dot{\gamma}(t)|^2 \, dt \le c ( 1 + {\cal A}^{[h]}(\gamma) ) 
   \quad \text{ and } \quad 
   |\gamma(t)| \le {\cal A}^{[h]}(\gamma) + c \quad \forall \, t \in [a, b]. $$
\end{lemma}

\begin{proof} 
If $V$ is bounded by $|V| \le C$, say, we clearly have 
$$ {\cal A}^{[h]}(\gamma) 
   \ge \frac{m}{2} \int_a^b |\dot{\gamma}(t)|^2 \, dt - C(b - a). $$

In case $V$ is unbounded, we define $\tilde{\gamma} : [a, b] \to \R^n$ by $\tilde{\gamma}(0) = \gamma(a)$, 
$$ \tilde{\gamma}(\tau_j) = \gamma(\tau_j), \quad 
   \tilde{\gamma}\left( \frac{3\tau_{j-1} + \tau_j}{4} \right) 
   = \tilde{\gamma}\left( \frac{\tau_{j-1} + 3\tau_j}{4} \right) 
   = \gamma\left(\frac{\tau_{j-1} + \tau_j}{2}\right) $$ 
for $j = 1, \ldots, l$ and affine interpolation. As in Lemma \ref{lemma:bounded-traj} we then find 
\begin{align*} 
  \int_a^b |\tilde{\gamma}(t)|^2 \, dt 
  \le 2 (b - a) R^2 + 2 (b - a)^2 \int_a^b |\dot{\gamma}(s)|^2 \, ds. 
\end{align*}
Now by construction $\tilde{\gamma}$ satisfies 
\begin{align*}
  \sum_{j = 1}^{l} V \left( \gamma \left(\frac{\tau_j + \tau_{j-1})}{2} \right) \right) (\tau_j - \tau_{j - 1}) 
  &= 2 \sum_{j = 1}^{l} \int_{\frac{3\tau_{j-1} + \tau_j}{4}}^{\frac{\tau_{j-1} + 3\tau_j}{4}} V(\tilde{\gamma}(t)) \, dt \\ 
  &\le 2 c_2 \int_a^b (1 + |\tilde{\gamma}(t)|^2) \, dt 
\end{align*}
as well as 
$$ \int_a^b |\dot{\gamma}(t)|^2 \, dt = \frac{1}{2} \int_a^b |\dot{\tilde{\gamma}}(t)|^2 \, dt. $$ 
So \begin{align*}
  {\cal A}^{[h]}(\gamma) 
  &\ge \frac{m}{4} \int_a^b |\dot{\tilde{\gamma}}(t)|^2 \, dt - 2 c_2 (b - a) - 2 c_2 \int_a^b |\tilde{\gamma}(t)|^2 \, dt \\ 
  &\ge \left( \frac{m}{4} - 4 c_2 (b - a)^2 \right) \int_a^b |\dot{\gamma}(t)|^2 \, dt - 2 c_2 (b - a) (1 + 2 R^2) \\ 
  &\ge \frac{m}{8}\int_a^b |\dot{\tilde{\gamma}}(t)|^2 \, dt - \sqrt{m c_2} (1 + R^2) \\ 
  &= \frac{m}{4}\int_a^b |\dot{\gamma}(t)|^2 \, dt - \sqrt{m c_2} (1 + R^2). 
\end{align*} 

This proves the first estimate. The second estimate then follows, possibly after enlarging $c$, precisely as at the end of the proof of Lemma \ref{lemma:bounded-traj}. \end{proof} 

The following lemma provides a well-known basic error estimate for piecewise affine interpolations.
\begin{lemma}\label{lemma:Ah-approx}
For every piecewise affine $\gamma$ subordinate to ${\cal T}_h$, 
\begin{align*}
  &\left| \int_a^b V(\gamma(t)) \, dt - \sum_{j = 1}^{l} V \left( \gamma \left(\frac{\tau_j + \tau_{j-1}}{2} \right) \right) (\tau_j - \tau_{j - 1})  \right| \\ 
  &\le h^2 \sup_{a \le t \le b} |\nabla^2 V(\gamma(t))| \int_a^b |\dot{\gamma}(t)|^2 \, dt. 
\end{align*}
\end{lemma}

\begin{proof} 
 This immediately follows from a Taylor expansion on every subinterval $(\tau_{j-1}, \tau_j)$, on which $\gamma$ is affine: 
\begin{align*}
  &\left| \int_{\tau_{j-1}}^{\tau_j} V(\gamma(t)) \, dt - V \left( \gamma \left( \frac{\tau_{j-1} + \tau_j}{2} \right) \right) (\tau_j - \tau_{j-1}) \right| \\ 
  &\le \frac{|\tau_j - \tau_{j-1}|^3}{24} \sup_{\tau_{j-1} \le t \le \tau_j} |\dot{\gamma}^T(t) \nabla^2 V(\gamma(t)) \dot{\gamma}(t)| \\ 
  &\le h^2 \sup_{\tau_{j-1} \le t \le \tau_j} |\nabla^2 V(\gamma(t))| \int_{\tau_{j-1}}^{\tau_j} |\dot{\gamma}(t)|^2 \, dt. \qedhere
\end{align*}
\end{proof} 

\begin{proof}[Proof of Theorem \ref{theo:Gamma-convergence-Numerics}] 
(i) We define ${\cal A}^R$ as in the proof of Theorem \ref{theo:Gamma-convergence-Eukl} and for given $k > 0$ choose $R$ so large that by Lemma \ref{lemma:bounded-traj-neu} 
$$ {\cal A}^{[h_N]}(\gamma) \le k 
   \implies |\gamma(t)| \le k + c \le R 
   \implies {\cal A}^R(\gamma) = {\cal A}(\gamma). $$ 
By Lemmas \ref{lemma:bounded-traj-neu} and \ref{lemma:Ah-approx} we obtain that for those $\gamma$ in addition 
$$ | {\cal A}^{[h_N]}(\gamma) - {\cal A}(\gamma) | 
   \le h^2 c(1 + k) \sup_{|x| \le k + c} |\nabla^2 V(x)|. $$
But then, if $\pi^{(N)}$ converges to $\pi$ in ${\cal P}(C([a, b]; M))$ weakly with uniformly bounded action, so that $\pr_t \# \pi = \mu_t$ for $t \in \{a, b\}$,  
\begin{align*}
  \int {\cal A}^{[h_N]} (\gamma) \, d \pi^N 
  \ge \int \min\{k, {\cal A}^{[h_N]} (\gamma)\} \, d \pi^N 
  \ge \int \left( \min\{k, {\cal A}^R (\gamma)\} - C h_N^2 \right) \, d \pi^N 
\end{align*}
for a constant $C = C(k)$. Letting first $N \to \infty$, $h_N \to 0$ and then $k \to \infty$, by the portmanteau theorem and monotone convergence we thus arrive at 
$$ \liminf_{N \to \infty} \int {\cal A}^{[h_N]} (\gamma) \, d \pi^N 
   \ge \int {\cal A} (\gamma) \, d \pi $$ 
as in the proof of Theorem \ref{theo:Gamma-convergence-Eukl}. 

(ii) In order to provide a recovery sequence for $\pi \in {\cal P}(C([a, b]; \R^n))$ we only need to observe that replacing the curves $\gamma_x$ constructed in the proof of Theorem \ref{theo:Gamma-convergence} with their piecewise affine interpolation subordinate to ${\cal T}_{h_N}$ only introduces negligible error terms in $d_{\rm BL}(\pi^{(N)}, \pi)$ and in $\mathbb{A}(\pi^{(N)})$ in the limit $h_N \to 0$. Still being supported on a common compact set, this defines a recovery sequence also for unbounded $V$, as shown in the proof of Theorem \ref{theo:Gamma-convergence-Eukl}. \end{proof}

\begin{proof}[Proof of Theorem \ref{theo:compactness-Numerics}]  
As ${\cal A}^{[h_N]}(\gamma) \le C$ by Lemma \ref{lemma:bounded-traj-neu} implies that $|\gamma(t)| \le C + c$ for $\pi^{(N)}$-a.e.\ $\gamma$, ${\cal A}^{[h_N]}$ is bounded from below on the joint support of the $\pi^{(N)}$. As in the proofs of Theorems \ref{theo:compactness} and \ref{theo:compactness-Eukl} we therefore have 
$$ \pi^{(N)} \left( \left\{ \gamma \in C([a, b]; \R^n) : {\cal A}^{[h_N]}(\gamma) \ge \frac{\tilde{C}}{\eps} \right\} \right) 
   \le \eps. $$ 
for some constant $\tilde{C} > 0$. As at the beginning of the proof of Theorem \ref{theo:Gamma-convergence-Numerics}, we see that Lemma \ref{lemma:bounded-traj-neu} implies $|{\cal A}^{[h]}(\gamma) - {\cal A}^R(\gamma)| \le 1$ for ${\cal A}^{[h]}(\gamma) \le \frac{\tilde{C}}{\eps}$ and sufficiently small $h$ and large $R$, so that Lemma \ref{lemma:A-coercive} yields a compact set $K_{\eps}$ with 
\begin{align*}
  \left\{ \gamma \in C([a, b]; \R^n) : {\cal A}^{[h_N]}(\gamma) \le \frac{\tilde{C}}{\eps} \right\} 
  &\subset \left\{ \gamma \in C([a, b]; \R^n) : {\cal A}^{R}(\gamma) \le \frac{\tilde{C}}{\eps} + 1 \right\} \\  
  &\subset K_{\eps} 
\end{align*}
and $\pi^{(N)}(K_{\eps}) \ge \pi^{(N)}(\{ \gamma \in C([a, b]; R) : {\cal A}^{[h_N]}(\gamma) \le \frac{\tilde{C}}{\eps}\}) \ge 1 - \eps$. 
\end{proof} 

We finally turn to the proof of Theorem \ref{theo:stationary-convergence-Numerics}. If $\gamma$ is a stationary point of ${\cal A}^{[h]}$, then it satisfies the discrete Euler-Lagrange equations 
\begin{align*}
  &m \frac{\gamma(\tau_j) - \gamma(\tau_{j-1})}{\tau_j - \tau_{j-1}} 
    - m \frac{\gamma(\tau_{j+1}) - \gamma(\tau_j)}{\tau_{j+1} - \tau_j} \\  
  &=\frac{\tau_j - \tau_{j-1}}{2} \nabla V \left( \frac{\gamma(\tau_{j-1}) + \gamma(\tau_j)}{2} \right) 
    + \frac{\tau_{j+1} - \tau_j}{2} \nabla V \left( \frac{\gamma(\tau_j) + \gamma(\tau_{j+1})}{2} \right), 
\end{align*}
$j = 1, \ldots, l-1$. 

\begin{lemma}\label{lemma:Stat-Ah-A}
Suppose $\gamma$ is a stationary point of ${\cal A}^{[h]}$ with $|\gamma(t)| \le R$ for all $t \in [a, b]$ and ${\cal A}^{[h]} \le C$. Let $\tilde{\gamma}$ be the extremal of ${\cal A}$ with $\tilde{\gamma}(a) = \gamma(a)$ and $\dot{\tilde{\gamma}}(a) = \frac{\gamma(a+h) - \gamma(a)}{h}$. Then there exists a constant $c = c(R, C)$ such that 
$$ |\gamma(t) - \tilde{\gamma}(t)| + |\dot{\gamma}(t) - \dot{\tilde{\gamma}}(t)| \le c h $$ 
for all $t \in [a, b]$. 
\end{lemma}

\begin{proof}
If $t \in (\tau_{j-1}, \tau_j)$, we have 
\begin{align*}
  \dot{\gamma}(t) 
  &= \frac{\gamma(\tau_1) - \gamma(\tau_0)}{\tau_1 - \tau_0} 
     + \sum_{i = 1}^{j - 1} \left( \frac{\gamma(\tau_{i+1}) - \gamma(\tau_i)}{\tau_{i + 1} - \tau_i} - \frac{\gamma(\tau_i) - \gamma(\tau_{i-1})}{\tau_i - \tau_{i - 1}} \right) \\ 
  &= \frac{\gamma(\tau_1) - \gamma(a)}{\tau_1 - a} 
   - \frac{1}{2m} \sum_{i = 1}^{j - 1} 
   \bigg( \nabla V \left( \frac{\gamma(\tau_{i+1}) + \gamma(\tau_i)}{2} \right) (\tau_{i+1} - \tau_i) \\ 
  &\qquad\qquad\qquad\qquad\qquad
   + \nabla V \left( \frac{\gamma(\tau_i) + \gamma(\tau_{i-1})}{2} \right) (\tau_i - \tau_{i-1})\bigg) \\ 
  &= \frac{\gamma(\tau_1) - \gamma(a)}{\tau_1 - a} 
   - \frac{1}{m} \int_a^t\nabla V(\gamma(s)) \, ds + O(h)
\end{align*}
by Lemma \ref{lemma:Ah-approx} with $V$ replaced by $\nabla V$. Note that by Lemma \ref{lemma:bounded-traj-neu}, the error term only depends on $C$ and $R$. Since $\tilde{\gamma}$ solves the continuous Euler-Lagrange equation, we also have 
\begin{align*}
  \dot{\tilde{\gamma}}(t) 
  = \dot{\tilde{\gamma}}(a) - \frac{1}{m} \int_a^t\nabla V(\tilde{\gamma}(s)) \, ds. 
\end{align*}
Then $\dot{\tilde{\gamma}}(a) = \frac{\gamma(a+h) - \gamma(a)}{h}$ and $\tilde{\gamma}(a) = \gamma(a)$ imply 
\begin{align*}
  |\dot{\gamma}(t) - \dot{\tilde{\gamma}}(t)| 
  &\le \frac{1}{m} \int_a^t |\nabla V(\gamma(s)) - \nabla V(\tilde{\gamma}(s))| \, ds + \tilde{C} h \\ 
  &\le \tilde{C} \int_a^t |\gamma(s) - \tilde{\gamma}(s)| \, ds + \tilde{C} h \\
  &\le \tilde{C} \int_a^t \int_a^s |\dot{\gamma}(r) - \dot{\tilde{\gamma}}(r)| \, dr \, ds + \tilde{C} h \\ 
  &= \tilde{C} \int_a^t (t - r) |\dot{\gamma}(r) - \dot{\tilde{\gamma}}(r)| \, dr + \tilde{C} h \\ 
  &\le \tilde{C} \int_a^t |\dot{\tilde{\gamma}}(r) - \dot{\gamma}(r)| \, dr + \tilde{C} h  
\end{align*}
for some $\tilde{C} > 0$ and all $a \le t \le t^* := \min\{\min\{s : |\tilde{\gamma}(s)| \ge R+1\}, b\}$. Gronwall's inequality now yields $|\dot{\tilde{\gamma}}(t) - \dot{\gamma}(t)| \le \tilde{C} h {\rm e}^{\tilde{C}(t - a)}$ and because of $\tilde{\gamma}(a) = \gamma(a)$ for suitable $c > 0$ we thus have
\[ |\dot{\tilde{\gamma}}(t) - \dot{\gamma}(t)| + |\tilde{\gamma}(t) - \gamma(t)| 
   \le c h \]
for these $t$. But then $|\tilde{\gamma}(t^*)| \le |\gamma(t^*)| + c h < R + 1$ for sufficiently small $h$ and so in fact $t^* = b$. 
\end{proof}

\begin{proof}[Proof of Theorem \ref{theo:stationary-convergence-Numerics}] 
As in the proof of Theorem \ref{theo:compactness-Numerics} we find 
\begin{align}\label{eq:AhN-Absch}
  \pi^{(N)} \left( \left\{ \gamma \in C([a, b]; \R^n) : {\cal A}^{[h_N]}(\gamma|_{[a, b']}) \ge c_{\eps} \right\} \right) 
  \le \eps 
\end{align} 
for given $\eps > 0$ and suitable $c_{\eps} > 0$. Consider the discrete Euler-Lagrange mapping $\phi^{L,h}_{\cdot - a}$, where $\phi^{L, h}_{\cdot - a} : \R^{2n} \to \R^{2n}$ maps $(x, v)$ to the solution of the discrete Euler-Lagrange equation $\gamma$ with $\gamma(a) = x$ and $\frac{\gamma(\tau_1) - \gamma(a)}{\tau_1 - a} = v$, so that with $v_i := \frac{\gamma_i(\tau_1) - \gamma_i(a)}{\tau_1 - a}$ for $\gamma^{(N)} = (\gamma_1, \ldots, \gamma_N)$ we can write 
$$ \gamma_i = \phi^{L, h_N}_{\cdot - a}(x_i, v_i) $$ 
and the measures 
$$ \eta^{(N)} 
   = \frac{1}{N} \sum_{i = 1}^N \delta_{(x_i, v_i)} $$
on $\R^{2n}$ satisfy $\pi^{(N)} = \pi_{\R^n} \phi^{L, h_N}_{\cdot - a} \# \eta^{(N)}$. By Lemmas \ref{lemma:stationary-prep} and \ref{lemma:Stat-Ah-A} (applied on the time interval $[a, b']$) there exists $c = c(\eps) > 0$ such that $|v_i| \le c$ for all $i$ with ${\cal A}^{[h_N]}(\gamma_i|_{[a, b']}) \ge c_{\eps}$, so that 
$$ \eta^{(N)}(\{(x, v) \in \R^{2n} : x \in K_a,~ |v| \le c\}) 
   \ge 1 - \eps $$ 
by \eqref{eq:AhN-Absch}. The sequence $\eta^{(N)}$ thus being tight, there exists a subsequence (not relabeled) such that $\eta^{(N)} \weakly \eta$. Noting that Lemma \ref{lemma:Stat-Ah-A} implies that $\phi^{L,h}_{\cdot - a} \to \phi^L_{\cdot - a}$ uniformly on compact subsets of $\R^{2n}$, we finally obtain that $\pi_{\R^n} \phi^{L,h}_{\cdot - a} \# \eta^{(N)} \weakly \pi_{\R^n} \phi^L_{\cdot - a} \# \eta$ and hence $\pi = \pi_\R^n \phi^L_{\cdot - a} \# \eta$. 
\end{proof}

\section*{Acknowledgments}

I am grateful to Michael Ortiz for pointing out this problem to me and for various interesting discussions on the subject.

 \typeout{References}

\end{document}